\documentclass[12pt,leqno]{article}

\usepackage{amsmath,amssymb,amsthm}
\usepackage[all]{xy}
\usepackage{lscape}

\makeatletter

\newtheorem{theorem}{Theorem}[section]
\newtheorem{proposition}[theorem]{Proposition}
\newtheorem{lemma}[theorem]{Lemma}
\newtheorem{corollary}[theorem]{Corollary}

\theoremstyle{definition}

\newtheorem{example}[theorem]{Example}

\theoremstyle{remark}

\theoremstyle{remark}
\newtheorem{remark}[theorem]{Remark}

\def\({{\rm (}}
\def\){{\rm )}}

\let\Mathrm\operator@font
\let\Cal\mathcal
\let\Bbb\mathbb
\newcommand{\fm}{\ensuremath{\mathfrak m}}

\def\standop#1{\mathop{\Mathrm #1}\nolimits}
\def\difstop#1#2{\expandafter\def\csname #1\endcsname{\standop{#2}}}
\def\defstop#1{\difstop{#1}{#1}}

\defstop{AB}
\defstop{ann}
\defstop{Ass}

\defstop{CMFI}
\defstop{codim}
\defstop{Coh}
\defstop{coht}
\defstop{Coker}
\defstop{Cone}
\defstop{Cl}
\defstop{Cox}

\defstop{depth}
\def\dq{/\!\!/}

\defstop{EM}
\defstop{embdim}
\defstop{End}
\defstop{ev}
\defstop{Ext}

\defstop{Flat}
\defstop{Func}

\def\GL{\text{\sl{GL}}}
\defstop{GD}
\defstop{Good}

\difstop{height}{ht}
\defstop{Hom}

\def\id{\mathord{\Mathrm{id}}}

\difstop{Image}{Im}
\defstop{ind}

\defstop{Ker}

\defstop{Lch}
\defstop{length}
\defstop{Lin}
\defstop{Lqc}
\defstop{lqc}

\defstop{Mat}
\defstop{Max}
\defstop{Min}
\defstop{Mod}
\defstop{Mor}
\defstop{MCM}

\defstop{Nerve}
\defstop{NonSFR}
\defstop{NonCMFI}
\defstop{NonNor}
\defstop{Nor}

\def\Orth{\text{\sl{O}}}

\defstop{PA}
\defstop{PM}
\defstop{Proj}
\defstop{Pic}

\defstop{Qch}
\defstop{qch}

\defstop{rad}
\defstop{rank}

\defstop{res}
\defstop{Reg}

\def\SL{\text{\sl{SL}}}
\def\SO{\text{\sl{SO}}}
\def\Sp{\text{\sl{Sp}}}
\defstop{Spec}
\defstop{supp}
\defstop{Supp}
\defstop{Sym}
\defstop{Sing}
\defstop{SFR}
\defstop{Soc}
\def\St{St}

\difstop{tdeg}{trans.deg}

\defstop{Tor}
\difstop{trace}{tr}

\def\ua{\underline{a}}

\defstop{Zar}

\def\O{\Cal O}

\def\fG{\mathfrak{G}}
\def\fS{\mathfrak{S}}
\def\fT{\mathfrak{T}}

\def\fm{\mathfrak{m}}

\def\uSpec{\mathop{\text{\underline{$\Mathrm Spec$}}}\nolimits}

\def\sdarrow#1{\downarrow\hbox to 0pt{\scriptsize$#1$\hss}}
\def\suarrow#1{\uparrow\hbox to 0pt{\scriptsize$#1$\hss}}
\def\ssearrow#1{\searrow\hbox to 0pt{\scriptsize$#1$\hss}}

\def\ext{{\textstyle\bigwedge}}

\def\section{\@startsection{section}{1}{\z@ }%
{-3.5ex plus -1ex minus -.2ex}{2.3ex plus .2ex}{\bf }}

\long\def\refname{\par\kern -3ex
\begin{center}\rm R\sc{eferences}\end{center}\par\kern 
-2ex}

\def\@seccntformat#1{\csname the#1\endcsname.\quad}

\def\@@@sect#1#2#3#4#5#6[#7]#8{%
   \ifnum #2>\c@secnumdepth 
      \def \@svsec {}\else \refstepcounter {#1}%
      \def\@svsec{}
   \fi 
   \@tempskipa #5\relax 
   \ifdim \@tempskipa >\z@ 
     \begingroup #6\relax \@hangfrom {\hskip #3\relax 
     \@svsec}{\interlinepenalty \@M #8\par }\endgroup 
     \csname #1mark\endcsname {#7}
   \else 
   \def \@svsechd {#6\hskip #3\@svsec #8\csname #1mark\endcsname {#7}}
   \fi \@xsect {#5}}

\def\@@@startsection#1#2#3#4#5#6{%
 \if@noskipsec \leavevmode \fi \par \@tempskipa #4\relax \@afterindenttrue 
 \ifdim \@tempskipa <\z@ \@tempskipa -\@tempskipa \@afterindentfalse 
 \fi \if@nobreak \everypar {}\else \addpenalty {\@secpenalty }\addvspace 
  {\@tempskipa }\fi \@ifstar {\@ssect {#3}{#4}{#5}{#6}}{\@dblarg 
  {\@@@sect {#1}{#2}{#3}{#4}{#5}{#6}}}}

\def\theparagraph{\thesection.\arabic{paragraph}}
\def\aparagraph{\@@@startsection{paragraph}{2}{\z@ }%
              {1.75ex plus .2ex minus .15ex}{-1em}{\bf(\theparagraph) } }
\def\paragraph{\@@@startsection{paragraph}{2}{\z@ }%
              {1.75ex plus .2ex minus .15ex}{-1em}{}{\bf(\theparagraph)} }

\let\c@theorem\c@paragraph

\title{Good filtrations and strong $F$-regularity of the ring of 
$U_P$-invariants}
\author{M{\sc itsuyasu} H{\sc ashimoto}}
\date{\normalsize
Graduate School of Mathematics, Nagoya University\\
Chikusa-ku,  Nagoya 464--8602 JAPAN\\
{\small \tt hasimoto@math.nagoya-u.ac.jp}}

\begin{document}

\maketitle

\begin{abstract}
Let $k$ be an algebraically closed field of positive characteristic, 
$G$ a reductive group over $k$, and $V$ a finite dimensional $G$-module.
Let $P$ be a parabolic subgroup of $G$, and $U_P$ its unipotent radical.
We prove that if $S=\Sym V$ has a good filtration, then $S^{U_P}$ is 
strongly $F$-regular.
\end{abstract}

\section{Introduction}

Throughout this paper,
$k$ denotes an algebraically closed field,
and $G$ a reductive group over $k$.
We fix a maximal torus $T$ and a Borel subgroup $B$ which contains $T$.
We fix a base $\Delta$ 
of the root system $\Sigma$ of $G$ so that $B$ is negative.
For any weight $\lambda\in X(T)$, we denote the induced module
$\ind_B^G(\lambda)$ by $\nabla_G(\lambda)$.
We denote the set of dominant weights by $X^+$.
For $\lambda\in X^+$, we call $\nabla_G(\lambda)$ the {\em dual Weyl module}
of highest weight $\lambda$.
Note that for $\lambda\in X(T)$, $\ind_B^G(\lambda)\neq 0$ if and only
if $\lambda\in X^+$ \cite[(II.2.6)]{Jantzen}, 
and if this is the case, $\nabla_G(\lambda)=
\ind_B^G(\lambda)$ is finite dimensional \cite[(II.2.1)]{Jantzen}.
We denote $\nabla_G(-w_0\lambda)^*$ by $\Delta_G(\lambda)$, and call it
the {\em Weyl module} of highest weight $\lambda$, where $w_0$ is the
longest element of the Weyl group of $G$.

We say that a $G$-module $W$ is {\em good}
if $\Ext^1_G(\Delta_G(\lambda),W)=0$ for any $\lambda\in X^+$.
A filtration 
$0=W_0\subset W_1\subset W_2\subset\cdots\subset W_r$
or $0=W_0\subset W_1\subset W_2\subset\cdots$ of $W$
is called a {\em good filtration} of $W$ if
$\bigcup_i W_i=W$, and for any $i\geq 1$, 
$W_i/W_{i-1}\cong \nabla_G(\lambda(i))$ for some $\lambda(i)\in X^+$.
A $G$-module $W$ has a good filtration if and only if $W$ is good and
of countable dimension \cite{Donkin2}.
See also \cite{Friedlander} and \cite[(III.1.3.2)]{Hashimoto5}.

Let $V$ be a finite dimensional $G$-module.
Let $P$ be a parabolic subgroup of $G$ containing $B$, and 
$U_P$ its unipotent radical.
The objective of this paper is to prove the following.

\begin{trivlist}\item[\bf Corollary~\ref{w-U_P.thm}]\it
Let $k$ be of positive characteristic.
Let $V$ be a finite dimensional $G$-module, and assume that
$S=\Sym V$ is good as a $G$-module.
Then $S^{U_P}$ is a finitely generated strongly $F$-regular Gorenstein UFD.
\end{trivlist}

An $F$-finite Noetherian ring $R$ of characteristic $p$ is said to be 
strongly $F$-regular if for any nonzerodivisor $a$ of $R$, there exists some
$r>0$ such that the
$R^{(r)}$-linear map $aF^r: R^{(r)}\rightarrow R$ ($x^{(r)}
\mapsto ax^{p^r}$) is
$R^{(r)}$-split \cite{HH}.
See (\ref{Frob-1.par}) for the notation.
A strongly $F$-regular $F$-finite ring is $F$-regular in the sense of 
Hochster--Huneke \cite{HH2}, and hence it is Cohen--Macaulay normal
(\cite[(4.2)]{HH3}, \cite{Kunz}, and \cite[(0.10)]{Velez}).

Under the same assumption as in Corollary~\ref{w-U_P.thm}, 
it has been known that 
$S^G$ is strongly $F$-regular \cite{Hashimoto}.
This old result is a corollary to our Corollary~\ref{w-U_P.thm}, since 
$T$ is linearly reductive and 
$S^G=S^B=(S^U)^T$ 
is a direct summand subring of $S^U$.
Under the same assumption as in Corollary~\ref{w-U_P.thm},
it has been proved that $S^U$ is $F$-pure \cite{Hashimoto2}.
An $F$-finite Noetherian ring $R$ of characteristic $p$ is said to be
$F$-pure if the Frobenius map $F:R^{(1)}\rightarrow R$ splits as 
an $R^{(1)}$-linear map.
Almost by definition, an $F$-finite 
strongly $F$-regular ring is $F$-pure, and hence
Corollary~\ref{w-U_P.thm}
(or Corollary~\ref{w-main.thm}) yields this old result, too.

Popov \cite{Popov2} proved that if the characteristic of $k$ is zero, 
$G$ is a reductive group over $k$, and
$A$ is a finitely generated $G$-algebra, then $A$ has rational singularities
if and only if $A^U$ does so.
Corollary~\ref{w-U_P.thm} (or Corollary~\ref{w-main.thm}) 
can be seen as a weak characteristic $p$ version of one direction of 
this result.
For a characteristic $p$ result related to 
the other direction, see Corollary~\ref{U-F-rational.thm}.

Section~2 is preliminaries.
We review the Frobenius twisting of rings, modules, and representations.
We also review the basics of $F$-singularities such as $F$-rationality
and $F$-regularity.

In Section~3, we study the ring theoretic properties of the invariant 
subring $k[G]^U$ of the coordinate ring $k[G]$.
The main results of this section are
Lemma~\ref{k[G]^U.thm} and Corollary~\ref{U-F-rational.thm}.

In Section~4, we state and prove our main result for $P=B$.
In order to do so, we introduce the notion of $G$-strong $F$-regularity
and $G$-$F$-purity.
These notions have already appeared in \cite{Hashimoto} essentially.
Our main theorem in the most general form 
can be stated using these words (Theorem~\ref{main.thm}).
As in \cite{Hashimoto}, Steinberg modules play important roles.

In Section~5, we generalize the main results in Section~4 to the case of
general $P$.
Donkin's results on $U_P$-invariants of good $G$-modules play an important
role here.

In Section~6, we give some examples.
The first one is the action associated with a finite quiver.
The second one is a special case of the first, and is a determinantal 
variety studied by De Concini and Procesi \cite{DP}.
The third one is also an example of the first.
It gives some new understandings on the study of
Goto--Hayasaka--Kurano--Nakamura \cite{GHKN}.
It also has some relationships with Miyazaki's study \cite{Miyazaki}.

In Section~7, we prove the following.

\begin{trivlist}
\item[\bf Theorem~\ref{good-open.thm}]\it
Let $S$ be a scheme, $G$ a reductive $S$-group acting trivially on
a Noetherian $S$-scheme $X$.
Let $M$ be a locally free coherent $(G,\O_X)$-module.
Then
\begin{multline*}
\Good(\Sym M)=
\{x\in X\mid \Sym (\kappa(x)\otimes_{\Cal O_{X,x}} M_x) \\
\text{ is a good $(\Spec \kappa(x)\times_S G)$-module}\},
\end{multline*}
and $\Good(\Sym M)$ is Zariski open in $X$.
\end{trivlist}

For a reductive group $G$ over a field 
which is not linearly reductive, there is
a finite dimensional $G$-module $V$ such that $(\Sym V)^G$ is not
Cohen--Macaulay \cite{Kemper}.
On the other hand, in characteristic zero, a reductive group $G$ is
linearly reductive, and Hochster and Roberts \cite{HR} proved that
$(\Sym V)^G$ is Cohen--Macaulay for any finite dimensional $G$-module $V$.
Later, Boutot proved that $(\Sym V)^G$ has rational singularities
\cite{Boutot}.
In view of Corollary~\ref{w-U_P.thm} and Theorem~\ref{good-open.thm},
it seems that the condition $\Sym V$ being good is an appropriate condition
to ensure that the good results in characteristic zero still holds.

\medskip
Acknowledgement: The author is grateful to Professor S. Donkin, 
Professor S.~Goto,
and
Professor V.~L.~Popov
for valuable advice.
Special thanks are due to Professor K.-i.~Watanabe for 
warm encouragement.

\section{Preliminaries}\label{prelimilaries}

\paragraph\label{Frob-1.par}
Throughout this paper, $p$ denotes a prime number.
Let $K$ be a perfect field of characteristic $p$.

For a $K$-space $V$ and $e\in\Bbb Z$, we denote the abelian group $V$ with
the new $K$-space structure $\alpha \cdot v= \alpha^{p^{-e}}v$ by $V^{(e)}$,
where the product of $\alpha^{p^{-e}}$ and $v$ in the right hand side is
given by the original $K$-space structure of $V$.
An element of $V$, viewed as an element of $V^{(e)}$ is sometimes denoted
by $v^{(e)}$ to avoid confusion.
Thus we have $v^{(e)}+w^{(e)}=(v+w)^{(e)}$ and 
$\alpha v^{(e)}=(\alpha^{p^{-e}}v)^{(e)}$.
If $f: V\rightarrow W$ is a $K$-linear map, then $f^{(e)}:V^{(e)}\rightarrow
W^{(e)}$ given by $f^{(e)}(v^{(e)})=w^{(e)}$ is a $K$-linear map again.
Note that $(?)^{(e)}$ is an autoequivalence of the category of $K$-vector
spaces.

If $A$ is a $K$-algebra, then $A^{(e)}$ with the 
multiplicative structure of $A$
is a $K$-algebra.
So $a^{(e)}b^{(e)}=(ab)^{(e)}$ for $a,b\in A$.
If $M$ is an $A$-module, then $M^{(e)}$ is an $A^{(e)}$-module
by $a^{(e)}m^{(e)}=(am)^{(e)}$.
For a $K$-algebra $A$ and $r\geq0$, the $r$th Frobenius map $F^r=
F^r_A:A\rightarrow A$ is defined by $F^r(a)=a^{p^r}$.
Then $F^r:A^{(r+e)}\rightarrow A^{(e)}$ is a $K$-algebra map for $e\in\Bbb Z$.
Note that $F^r(a^{(r+e)})=(a^{p^r})^{(e)}$.
$F^r:A^{(r+e)}\rightarrow A^{(e)}$ is also written as $(F^r)^{(e)}$.

In commutative ring theory, $A^{(e)}$ is sometimes denoted by
${}^{-e}A$ or $A^{p^e}$.

\paragraph\label{Frob-2.par}
For a $K$-scheme $X$, the scheme $X$ with the new $K$-scheme structure
$X\xrightarrow{f}\Spec K\xrightarrow{{}^a(F^{-e}_K)}\Spec K$ is denoted
by $X^{(e)}$, where $f$ is the original structure map of $X$ as a 
$K$-scheme.
So for a $K$-algebra $A$, $\Spec A^{(e)}$ is identified with $(\Spec A)^{(e)}$.
The Frobenius map $F^r: X\rightarrow X^{(r)}$ is a $K$-morphism.
Note that $(?)^{(e)}$ is an autoequivalence of the category of $K$-schemes
with the quasi-inverse $(?)^{(-e)}$, and it preserves the product.
So the canonical map $(X\times Y)^{(e)}\rightarrow X^{(e)}\times
Y^{(e)}$ is an isomorphism.
If $G$ is a $K$-group scheme, then with the product $G^{(e)}\times G^{(e)}
\cong (G\times G)^{(e)}\xrightarrow{\mu^{(e)}} G^{(e)}$, $G^{(e)}$ is a
$K$-group scheme, and $F^r: G^{(e)}\rightarrow G^{(e+r)}$ is a homomorphism
of $K$-group schemes.
If $V$ is a $G$-module, then $V^{(e)}$ is a $G^{(e)}$-module in a natural way.
Thus $V^{(r)}$ is a $G$-module again for $r\geq 0$ via $F^r:G\rightarrow 
G^{(r)}$.
If $V$ has a basis $v_1,\ldots,v_n$, $g\in G(K)$, and
$gv_j=\sum_i c_{ij}v_i$, then $gv_j^{(r)}=\sum_i c_{ij}^{p^r}v_i^{(r)}$.
If $A$ is a $G$-algebra, then $A^{(r)}$ is a $G$-algebra again.
If $M$ is a $(G,A)$-module, then $M^{(r)}$ is a $(G,A^{(r)})$-module.
See \cite{Hashimoto}.

\paragraph
Let $A$ be an $\Bbb F_p$-algebra.
We say that $A$ is {\em $F$-finite} if $A$ is a finite $A^{(1)}$-module.
An $F$-finite Noetherian $K$-algebra is excellent \cite{Kunz}.

Let $A$ be Noetherian.
We denote by $A^\circ$ the set $A\setminus\bigcup_{P\in\Min A}P$, 
where $\Min A$ denotes the set of minimal primes of $A$.
Let $M$ be an $A$-module and $N$ its submodule.
We define
\begin{multline*}
\Cl_A(N,M)=N_M^*:=\{x\in M \mid \exists c\in A^\circ\;\exists e_0\geq 1\;
\forall e\geq e_0\\
x\otimes c^{(-e)}\in M/N\otimes_AA^{(-e)} 
\text{ is zero}\},
\end{multline*}
and call it the {\em tight closure} of $N$ in $M$.
Note that $\Cl_A(N,M)$ is an $R$-submodule of $M$ containing $N$
\cite[Section~8]{HH2}.
We say that $N$ is {\em tightly closed} in $M$ if $\Cl_A(N,M)=N$.
For an ideal $I$ of $A$, $\Cl_A(I,A)$ is simply denoted by $I^*$.
If $I^*=I$, then we say that $I$ is tightly closed.

We say that $A$ is {\em very strongly $F$-regular} if 
for any $a\in A^\circ$, there exists some $r\geq 1$ such that the 
$A^{(r)}$-linear map $aF^r_A:A^{(r)}\rightarrow A$ has a splitting.
That is, there is an $A^{(r)}$-linear map $\Phi:A\rightarrow A^{(r)}$
such that $\Phi aF^r=\id_{A^{(r)}}$ \cite{HH3}, \cite{Hashimoto8}.
We say that $A$ is {\em strongly $F$-regular} if $\Cl_A(N,M)=N$ for any
$A$-module $M$ and its submodule $N$ \cite[p.~166]{Hochster2}.
We say that $A$ is {\em weakly $F$-regular} if $I=I^*$ for any ideal $I$ of $A$
\cite{HH2}.
We say that $A$ is {\em $F$-regular} if for any prime ideal $P$ of $A$,
$A_P$ is weakly $F$-regular \cite{HH2}.
We say that $A$ is {\em $F$-rational} if $I=I^*$
for any ideal $I$ generated by
$\height I$ elements, where $\height I$ denotes the height of $I$.

\begin{lemma}\label{F-sing.thm}
Let $A$ be a Noetherian $\Bbb F_p$-algebra.
\begin{description}
\item[(i)] If $A$ is very strongly $F$-regular, then it is strongly 
$F$-regular.
The converse is true, if $A$ is either local, $F$-finite, or 
essentially of finite type over an excellent local ring.
\item[(ii)] If $A$ is strongly $F$-regular, then it is $F$-regular.
An $F$-regular ring is weakly $F$-regular.
A weakly $F$-regular ring is $F$-rational.
\item[(iii)] A pure subring of a strongly $F$-regular ring is strongly
$F$-regular.
\item[(iv)] An $F$-rational ring is normal.
\item[(v)] An $F$-rational ring which is a homomorphic image of a 
Cohen--Macaulay ring is Cohen--Macaulay.
\item[(vi)] A locally excellent $F$-rational ring is Cohen--Macaulay.
\item[(vii)] If $A=\bigoplus_{i\geq 0}A_i$ is graded and $A_0$ is
a field, and if 
$A$ is weakly $F$-regular, then $A$ is very strongly $F$-regular.
\item[(viii)] A Gorenstein $F$-rational ring is strongly $F$-regular.
\end{description}
\end{lemma}

\begin{proof} {\bf(i)} is \cite[(3.6), (3.9), (3.35)]{Hashimoto8}.
{\bf(ii)} is \cite[(3.7)]{Hashimoto8}, \cite[(4.15)]{HH2}, and
\cite[(4.2)]{HH3}.
{\bf (iii)} is \cite[(3.17)]{Hashimoto8}.
{\bf (iv)} and {\bf(v)} are \cite[(4.2)]{HH3}.
{\bf (vi)} is \cite[(0.10)]{Velez}.

{\bf (vii)} is \cite[(4.3)]{LS}, if the field $A_0$ is $F$-finite.
We prove the general case.
By \cite[(4.15)]{HH2}, $A_\fm$ is weakly $F$-regular,
where $\fm=\bigoplus_{i>0}A_i$ is the irrelevant ideal.
Let $K$ be the perfect closure (the largest purely inseparable extension) 
of $A_0$, and set $B:=K\otimes_{A_0}A$.
Then $B$ is purely inseparable over $A$.
It is easy to see that $B_\fm:=B\otimes_A A_\fm$ is a local ring whose maximal
ideal is $\fm B_\fm$.
By \cite[(6.17)]{HH3}, $B_\fm$ is weakly $F$-regular.
By the proof of \cite[(4.3)]{LS}, $B_\fm$ and $B$ are strongly $F$-regular.
By \cite[(3.17)]{Hashimoto8}, $A$ is strongly $F$-regular.
As $A$ is finitely generated over the field $A_0$, $A$ is 
very strongly $F$-regular by {\bf (i)}.

{\bf (viii)} Let $A$ be a Gorenstein $F$-rational ring.
By \cite[(4.2)]{HH3}, $A_\fm$ is Gorenstein $F$-rational for any maximal
ideal $\fm$ of $A$.
If $A_\fm$ is strongly $F$-regular for any maximal ideal $\fm$ of $A$, then
$A$ is strongly $F$-regular by \cite[(3.6)]{Hashimoto8}.
Thus we may assume that $(A,\fm)$ is local.
Let $(x_1,\ldots,x_d)$ be a system of parameters of $A$.
Then an element of $H^d_\fm(A)$ as the $d$th cohomology group of the modified
\v Cech complex \cite[(3.5)]{BH} is of the form $a/(x_1\cdots x_d)^t$ for some
$t\geq 0$ and $a\in A$.
This element is zero if and only if $a\in (x_1^t,\ldots,x_d^t)$, by
\cite[(10.3.20)]{BH}.
So this element is in the tight closure $(0)^*_{H^d_\fm(A)}$ of $0$
if and only if
$a\in(x_1^t,\ldots,x_d^t)^*=(x_1^t,\ldots,x_d^t)$, and hence
$(0)^*_{H^d_\fm(A)}=0$.
As $A$ is Gorenstein, $H^d_\fm(A)$ is isomorphic to the injective hull
$E_A(A/\fm)$ of the residue field $A/\fm$.
By \cite[(3.6)]{Hashimoto8}, $A$ is strongly $F$-regular.
\end{proof}

\paragraph Let $K$ be a field of characteristic zero, and $A$ a $K$-algebra
of finite type.
We say that $A$ is of {\em strongly $F$-regular type} if there is a 
finitely generated $\Bbb Z$-subalgebra $R$ of $A$ and a finitely generated
flat $R$-algebra $A_R$ such that $A\cong K\otimes_R A_R$, and
for any maximal ideal $\frak m$ of $R$, $R/\frak m\otimes_R A_R$ is
strongly $F$-regular.
See \cite[(2.5.1)]{Hara}.

\section{The invariant subring $k[G]^U$}\label{cox.sec}

\paragraph
Let the notation be as in the introduction.
Let $\Lambda$ be an abelian group.
We say that $A=\bigoplus_{\lambda\in\Lambda} A_\lambda$ is a 
{\em $\Lambda$-graded $G$-algebra}
if $A$ is both a $G$-algebra and a $\Lambda$-graded $k$-algebra, 
and each $A_\lambda$ is
a $G$-submodule of $A$ for $\lambda\in\Lambda$.
This is the same as to say that $A$ is a $G\times \Spec k\Lambda$-algebra,
where $k\Lambda$ is the group algebra of $\Lambda$ over $k$.
It is a commutative cocommutative Hopf algebra with each $\lambda\in\Lambda$
group-like.

We say that a $\Bbb Z$-graded $k$-algebra $A=\bigoplus_{i\in \Bbb Z}A_i$ is
{\em positively graded} if $A_i=0$ for $i<0$ and $k\cong A_0$.

\paragraph\label{P-and-G.par}
Let the notation be as in the introduction.

We need to review Popov--Grosshans filtration \cite{Popov}, \cite{Grosshans2}.

Let us fix (until the end of this section) a function
$h:X(T)\rightarrow \Bbb Z$ such that (i) $h(X^+)\subset\Bbb N=\{0,1,\ldots\}$;
(ii) $h(\lambda)> h(\mu) $ whenever $\lambda>\mu$;
(iii) $h(\chi)=0$ for $\chi\in X(G)$.
Such a function $h$ exists \cite[Lemma~6]{Grosshans2}.

Let $V$ be a $G$-module.
For a poset ideal $\pi$ of $X^+$, we define $O_\pi(V)$ to be
the sum of all the $G$-submodules $W$ of $V$ such that
$W$ belongs to $\pi$, that is, if $\lambda\in X^+$ and $W_\lambda\neq 0$,
then $\lambda\in\pi$.
$O_\pi(V)$ is the biggest $G$-submodule of $V$ belonging to $\pi$.
We set $\pi(n):=h^{-1}(\{0,1,\ldots,n\})$ for $n\geq 0$ and $\pi(n)=\emptyset$
for $n<0$.
We also define $V\langle n\rangle:=O_{\pi(n)}(V)$.

For a $G$-algebra $A$, $(A\langle n\rangle)$ is a filtration of $A$.
That is, $1\in A\langle 0\rangle\subset A\langle1\rangle\subset\cdots$, 
$\bigcup_n A\langle n\rangle=A$, and 
$A\langle n\rangle\cdot A\langle m\rangle
\subset A\langle n+m\rangle$.
The Rees ring $\Cal R(A)$ of $A$ is the subring
$\bigoplus_n A\langle n\rangle t^n$ of $A[t]$.
Letting $G$ act on $t$ trivially, $A[t]$ is a $G$-algebra, and
$\Cal R(A)$ is a $G$-subalgebra of $A[t]$.
So the associated graded ring $\Cal G(A):=\Cal R(A)/t\Cal R(A)$ is
also a $G$-algebra.

We denote the opposite of $U$ by $U^+$.

\begin{theorem}[Grosshans {\cite[Theorem~16]{Grosshans}}]\label{grosshans.thm}
Let $A$ be a $G$-algebra which is good as a $G$-module.
There is a $G$-algebra isomorphism
$\Phi:\Cal G(A)\rightarrow (A^{U^+}\otimes k[G]^U)^T$,
where $U$ acts right regularly on $k[G]$, $T$ acts right regularly on 
$k[G]^U$ \(because $T$ normalizes $U$\), 
and $G$ acts left regularly on $k[G]^U$
and trivially on $A^{U^+}$.
\qed
\end{theorem}

The direct product $G\times G$ acts on the coordinate ring $k[G]$ by
\[
((g_1,g_2)f)(g)=f(g_1^{-1}gg_2)\qquad
(f\in k[G],\;g,g_1,g_2\in G(k)).
\]
In particular, $k[G]$ is a $G\times B$-algebra.
Taking the invariant subring by the subgroup $U=\{e\}\times U\subset 
G\times B$, $k[G]^U$ is a $G\times T$-algebra, since $T$ normalizes
$U$.
Thus $k[G]^U=\bigoplus_{\lambda\in X(T)}k[G]^U_\lambda$ is an
$X(T)$-graded $G$-algebra.
As a $G$-module,
\[
k[G]^U_\lambda=\{f\in k[G]\mid f(gb)=\lambda(b)f(g)\}\cong 
((-\lambda)\otimes k[G])^B=\ind_B^G(-\lambda)
\]
for $\lambda\in X(B)=X(T)$ by the definition of induction, see 
\cite[(I.3.3)]{Jantzen}.
Thus we have:

\begin{lemma}\label{decompose-k[G]^U.thm}
$k[G]^U\cong \bigoplus_{\lambda\in X^+}\nabla_G(\lambda)\boxtimes (-\lambda)$ 
as a $G\times T$-module.
It is an integral domain.
\qed
\end{lemma}

The converse is also true.

\begin{lemma}\label{G-T.thm}
Let $A$ be a $G\times T$-algebra such that $A\cong \bigoplus_{\lambda
\in X^+}\nabla_G(\lambda)\boxtimes(-\lambda)$ as a $G\times T$-module
and that $A^{U^+}$ is a domain, where $U^+=U^+\times\{e\}\subset G\times T$.
Then $A\cong k[G]^U$ as a $G\times T$-algebra.
\end{lemma}

\begin{proof}
Let $\varphi: X^+\rightarrow X(T)\times X(T)$ be the semigroup 
homomorphism given by $\varphi(\lambda)=(\lambda,-\lambda)$.
$A^{U^+}$ is a $\varphi(X^+)$-graded domain, and each homogeneous component
$A^{U^+}_{\varphi\lambda}=\nabla_G(\lambda)^{U^+}\boxtimes(-\lambda)$ is
one-dimensional.
So by \cite[Lemma~5.5]{Hashimoto4}, 
$A^{U^+}\cong k\varphi(X^+)$ as an $X(T)\times X(T)$-graded $k$-algebra.

Set $G'=G\times T$, and $T'=T\times T$.
Define $h': X(T')\cong X(T)\times X(T)\rightarrow \Bbb Z$ by
$h'(\lambda,\mu)=h(\lambda)$.
For a $G'$-algebra $A'$, we have a filtration of $A'$ from $h'$
as in (\ref{P-and-G.par}).
We denote the associated graded algebra by $\Cal G'(A')$.

It is easy to see that $A\cong\Cal G'(A)$, and this is isomorphic to
$B:=(k\varphi(X^+)\otimes (k[G]^U\boxtimes k[T]))^{T'}$ 
by Theorem~\ref{grosshans.thm} (applied to $G'$).
We define $\psi:k[G]^U\rightarrow B$ by
\[
\psi(a\otimes t_{-\lambda})=
(t_\lambda\otimes t_{-\lambda})\otimes (a\otimes t_{-\lambda})\otimes
t_{\lambda},
\]
where $a\in\nabla_G(\lambda)$ and for $\mu\in X(T)$, $t_\mu$ is the
element $\mu$ considered as a basis element of $kX(T)=k[T]$.
We consider that $\nabla_G(\lambda)\boxtimes(-\lambda)\subset k[G]^U$.
With respect to the left regular action, $t_\lambda$ is of weight $-\lambda$.
So $\psi$ is a $G\times T$-algebra isomorphism.
Thus $A\cong k[G]^U$ as a $G\times T$-algebra, as desired.
\end{proof}

Assume that $G$ is semisimple simply connected.
Then by \cite{Popov}, $X(B)\rightarrow \Pic(G/B)$ $(\lambda\mapsto
\Cal L(\lambda))$ is an isomorphism, where $\Cal L(\lambda)=\Cal L_{G/B}(
\lambda)$ is the
$G$-linearized invertible sheaf on the flag variety $G/B$,
associated to the $B$-module $\lambda$, see \cite[(I.5.8)]{Jantzen}.
Thus we have

\begin{lemma}\label{cox.thm}
If $G$ is semisimple and simply connected, then
the Cox ring \(the total coordinate ring, see {\rm\cite{Cox}}, 
{\rm\cite{EKW}}\)
$\Cox(G/B)$ is isomorphic to $\bigoplus_{\lambda\in X^+}\nabla_G(\lambda)$,
as an $X(T)$-graded $G$-module \(that is, $G\times T$-module\), where 
both $H^0(G/B,\Cal L(\lambda))\subset\Cox(G/B)$
and $\nabla_G(\lambda)$ are assigned degree $-\lambda$.
The Cox ring 
$\Cox(G/B)$ is also an integral domain, and hence isomorphic to $k[G]^U$ as
an $X(T)$-graded $G$-algebra.
\end{lemma}

\begin{proof}
The first assertion follows from the fact that 
$H^0(G/B,\Cal L(\lambda))=\nabla_G(\lambda)$ for $\lambda
\in X^+$ and $H^0(G/B,\Cal L(\lambda))=0$ for $\lambda\in X(B)\setminus X^+$
\cite[(II.2.6)]{Jantzen}.

Consider the $\Cal O_{G/B}$-algebra $\Cal S:=\Sym(\Cal L(\lambda_1)\oplus
\cdots\oplus\Cal L(\lambda_l))$, where $\lambda_1,\ldots,\lambda_l$ are
the fundamental dominant weights.
Being a vector bundle over $G/B$, $\uSpec\Cal S$ is integral.
Hence $\Cox(G/B)\cong\Gamma(G/B,\Cal S)$ is a domain.

The last assertion follows from Lemma~\ref{G-T.thm}.
\end{proof}

\begin{lemma}\label{UFD.thm}
If $G$ is semisimple and simply connected, then $k[G]^U$ is a UFD.
\end{lemma}

\begin{proof}
This is a consequence of Lemma~\ref{cox.thm} and
\cite[Corollary~1.2]{EKW}.

There is another proof.
Popov \cite{Popov} proved that $k[G]$ is a UFD.
Moreover, $G$ does not have a nontrivial character, since $G=[G,G]$, see
\cite[(29.5)]{Humphreys}.
It follows easily that $k[G]^\times=k^\times$ by \cite[Theorem~3]{Rosenlicht}.
As $U$ is unipotent, $U$ does not have a nontrivial character.
The lemma follows from Remark~3 after Proposition~2 of \cite{Popov3}.
See also \cite[(4.31)]{UFD}.
\end{proof}

By \cite[(2.1)]{Grosshans}, $k[G]^U$ is finitely generated.
See also \cite{RR} and \cite[Theorem~9]{Grosshans2}.

By \cite[Lemma~5.6]{Hashimoto4} and
Lemma~\ref{decompose-k[G]^U.thm}, $k[G]^U$ is strongly $F$-regular
in positive characteristic, and strongly $F$-regular type in
characteristic zero.
In any characteristic, $k[G]^U$ is Cohen--Macaulay normal.

In any characteristic, if $G$ is semisimple simply connected,
being a finitely generated Cohen--Macaulay UFD, $k[G]^U$ is
Gorenstein \cite{Murthy}.

Combining the observations above, we have:

\begin{lemma}\label{k[G]^U.thm}
$k[G]^U$ is finitely generated.
It is strongly $F$-regular in positive characteristic,
and strongly $F$-regular type in characteristic zero.
If $G$ is semisimple and simply connected,
then $k[G]^U$ is a Gorenstein UFD.
\end{lemma}

\begin{corollary}\label{U-F-rational.thm}
Let $k$ be of positive characteristic, and
$A$ be a $G$-algebra which is good as a $G$-module.
If $A^{U}$ is finitely generated and strongly $F$-regular,
then $A$ is finitely generated and $F$-rational.
\end{corollary}

\begin{proof}
This is proved similarly to 
\cite[Proposition~10]{Popov2} and \cite[Theorem~17]{Grosshans2}.

As $U$ and $U^+$ are conjugate, $A^{U^+}\cong A^U$, and it is
finitely generated and strongly $F$-regular by assumption.
Note that $A$ is finitely generated \cite[Theorem~9]{Grosshans2}.

Note that $k[G]^U$ is finitely generated and strongly $F$-regular
by Lemma~\ref{k[G]^U.thm}.
So the tensor product $A^{U^+}\otimes k[G]^U$ is finitely generated and
strongly $F$-regular
by \cite[(5.2)]{Hashimoto4}.
Thus its direct  summand subring $(A^{U^+}\otimes k[G]^U)^T$ is also
finitely generated and strongly $F$-regular \cite[(3.1)]{HH}.
By Theorem~\ref{grosshans.thm}, $\Cal G(A)$ is finitely generated and 
strongly $F$-regular, hence
is $F$-rational (\cite[(3.1)]{HH} and \cite[(4.2)]{HH3}).
By \cite[(7.14)]{HM}, $A$ is $F$-rational.
\end{proof}

\section{The main result}

Let the notation be as in the introduction.
In this section, the characteristic of $k$ is $p>0$.

\paragraph
For a $G$-module $W$ and $r\geq 0$, $W^{(r)}$ denotes the $r$th 
Frobenius twist of $W$, see Section~\ref{prelimilaries}
and \cite[(I.9.10)]{Jantzen}.

Let $\rho$ denote the half sum of positive roots.
For $r\geq 0$, let $\St_r$ denote the $r$th Steinberg module
$\nabla_G((p^r-1)\rho)$, if $(p^r-1)\rho$ is a weight of $G$.
Note that $(p^r-1)\rho$ is a weight of $G$ if $p$ is odd or
$[G,G]$ is simply connected.

The following lemma, which is the dual assertion of 
\cite[Theorem~3]{Hashimoto}, follows immediately from
\cite[(10.6)]{Jantzen}.

\begin{lemma}\label{steinberg.thm}
Let $(p^r-1)\rho$ be a weight of $G$ for any $r\geq 0$.
Let $V$ be a finite dimensional $G$-module.
Then there exists some $r_0\geq 1$ such that for any $r\geq r_0$ and
any subquotient $W$ of $V$, any nonzero \(or equivalently, surjective\)
$G$-linear map $\varphi:\St_r\otimes W\rightarrow \St_r$ admits a 
$G$-linear map $\psi:\St_r\rightarrow \St_r\otimes W$ such that
$\varphi\psi=\id_{\St_r}$.
\end{lemma}

We set $\tilde G=\rad G\times \Gamma$, where $\rad G$ is the radical of $G$,
and $\Gamma\rightarrow [G,G]$ is the universal covering of the
derived subgroup $[G,G]$ of $G$.
Note that there is a canonical surjective 
map $\tilde G\rightarrow G$, and hence
any $G$-module (resp.\ $G$-algebra) 
is a $\tilde G$-module (resp.\ $\tilde G$-algebra) in a natural way.
The restriction functor $\res_{\tilde G}^G$ is full and faithful.

Let $S=\bigoplus_{i\geq 0}S_i$ be a positively graded finitely generated
$G$-algebra which is an integral domain.

Assume first that $(p^r-1)\rho$ is a weight of $G$ for $r\geq 0$.
We say that $S$ is {\em $G$-strongly $F$-regular} if for any
nonzero homogeneous element $a$ of $S^G$,
there exists some $r\geq 1$ such that
the $(G,S^{(r)})$-linear map
\[
\id\otimes aF^r: \St_r\otimes S^{(r)}\rightarrow \St_r\otimes S
\qquad
(x\otimes s^{(r)}\mapsto x\otimes as^{p^r})
\]
is a split mono.
In general, we say that  $S$ is $G$-strongly $F$-regular if it is so
as a $\tilde G$-algebra.

The following is essentially proved in \cite{Hashimoto}.
We give a proof for completeness.

\begin{lemma}\label{gsfr-sfr.thm}
If $S$ is a $G$-strongly $F$-regular
positively graded finitely generated $G$-algebra domain, 
then $S^G$ is strongly $F$-regular.
\end{lemma}

\begin{proof}
We may assume that $G=\tilde G$.
Let $A:=S^G$.

As we assume that $S$ is a finitely generated 
positively graded domain, $A$ is a finitely generated positively graded
domain, see \cite[Appendix to Chapter~1, {\bf A}]{GIT}.
Let $a$ be a nonzero homogeneous element of $A$ such that
$A[1/a]$ is regular.
Take $r\geq 1$ so that 
$\id\otimes aF^r: \St_r\otimes S^{(r)}\rightarrow \St_r\otimes S$
is a split mono.
Let $\Phi:\St_r\otimes S\rightarrow \St_r\otimes S^{(r)}$ be 
a $(G,S^{(r)})$-linear map such that $\Phi\circ(\id\otimes aF^r)=\id$.
Then consider the commutative diagram of $(G,A^{(r)})$-modules
\[
\xymatrix{
\St_r\otimes A^{(r)} \ar@{^{(}->}[r] \ar[d]^{aF^r} &
\St_r\otimes S^{(r)} \ar[r]^\id \ar[d]^{aF^r} &
\St_r\otimes S^{(r)} \\
\St_r\otimes A \ar@{^{(}->}[r] &
\St_r\otimes S. \ar[ur]^\Phi
}
\]
Then applying the functor $\Hom_G(\St_r,?)$ to this diagram, we get
the commutative diagram of $A^{(r)}$-modules
\[
\xymatrix{
A^{(r)} \ar[r]^\id \ar[d]^{aF^r} &
A^{(r)} \ar[r]^\id \ar[d] &
A^{(r)} \\
A \ar[r] &
\Hom_G(\St_r,\St_r\otimes S), \ar[ur]
}
\]
see \cite[Proposition~1, {\bf 5}]{Hashimoto}.
This shows that the $A^{(r)}$-linear map 
$aF^r:A^{(r)}\rightarrow A$ splits.
By \cite[(3.3)]{HH}, $A$ is strongly $F$-regular.
\end{proof}

The following is also proved in \cite{Hashimoto} (see the proof of
\cite[Theorem~6]{Hashimoto}).

\begin{theorem}\label{good-filtration.thm}
Let $V$ be a finite dimensional $G$-module.
If $S=\Sym V$ has a good filtration
\(see the introduction for definition\),
then $S$ is $G$-strongly $F$-regular.
\end{theorem}

\begin{lemma}
Let $S$ be 
a $G$-strongly $F$-regular
positively graded finitely generated $G$-algebra domain, 
and assume that there exists some $a\in S^G\setminus\{0\}$ such that
$S[1/a]$ is strongly $F$-regular.
Then $S$ is strongly $F$-regular.
\end{lemma}

\begin{proof}
We may assume that $G=\tilde G$.
Let $I$ be the radical ideal of $S$ which defines the non-strongly
$F$-regular locus of $S$.
Such an ideal exists, see \cite[(3.3)]{HH}.
Then $I$ is $G\times \Bbb G_m$-stable, and hence $I\cap S^G$ is
$\Bbb G_m$-stable.
In other words, $I\cap S^G$ is a homogeneous ideal of $S^G$.
By assumption, $0\neq a\in I\cap S^G$.
So $I\cap S^G$ contains a nonzero homogeneous element $b$.
Take $r\geq 1$ so that $1\otimes bF^r:\St_r\otimes S^{(r)}\rightarrow
\St_r\otimes S$ has a spitting.
Let $x$ be any nonzero element of $\St_r$.
Then $x\otimes \id:S^{(r)}\cong k\otimes S^{(r)}\rightarrow \St_r\otimes
S^{(r)}$ given by $s^{(r)}\mapsto x\otimes s^{(r)}$ is
a split mono as an $S^{(r)}$-linear map.
Thus $(x\otimes \id_S)(bF^r)=(\id\otimes bF^r)(x\otimes \id_{S^{(r)}})$ 
is a split mono as an $S^{(r)}$-linear map, 
and hence so is $bF^r:S^{(r)}\rightarrow S$.
By \cite[(3.3)]{HH}, $S$ is strongly $F$-regular.
\end{proof}

Let $S$ be a finitely generated $G$-algebra.
We say that $S$ is {\em $G$-$F$-pure} if there exists some $r\geq 1$
such that $\id \otimes F^r:\St_r\otimes S^{(r)}\rightarrow \St_r\otimes
S$ splits as a $(G,S^{(r)})$-linear map.
Obviously, a $G$-strongly $F$-regular finitely generated positively graded
$G$-algebra domain is $G$-$F$-pure.
The following is essentially proved in \cite{Hashimoto2}.

\begin{lemma}
Let $S$ be a $G$-$F$-pure finitely generated $G$-algebra.
Then $S^G$ is $F$-pure.
\end{lemma}

\begin{proof} This is proved similarly 
to Lemma~\ref{gsfr-sfr.thm}.
See also \cite{Hashimoto2}.
\end{proof}

\begin{lemma}\label{tensor-G-F-pure.thm}
Let $S$ and $S'$ be a $G$-$F$-pure finitely generated $G$-algebras.
Then the tensor product $S\otimes S'$ is $G$-$F$-pure.
\end{lemma}

\begin{proof}
This is easy, and we omit the proof.
\end{proof}

\begin{lemma}\label{multiple.thm}
Let $S$ be a $G$-$F$-pure finitely generated $G$-algebra,
and assume that the $(G,S^{(r)})$-linear map
\[
\id\otimes F^r:\St_r\otimes S^{(r)}\rightarrow \St_r\otimes S
\]
splits.
Then the $(G,S^{(nr)})$-linear map
\[
\id \otimes F^{nr}:\St_{nr}\otimes S^{(nr)}\rightarrow \St_{nr}\otimes S
\]
splits for any $n\geq 0$.
\end{lemma}

\begin{proof}
Induction on $n$.
The case that $n=0$ is trivial.
Assume that $n>0$.
Note that $\St_{nr}\cong \St_r\otimes \St_{(n-1)r}^{(r)}$.
So
\[
\id\otimes (F^{(n-1)r})^{(r)}:
\St_{nr}\otimes S^{(nr)}\rightarrow \St_{nr}\otimes S^{(r)}
\]
is identified with the map
\[
\id\otimes (\id\otimes F^{(n-1)r})^{(r)}:
\St_r\otimes(\St_{(n-1)r}\otimes S^{((n-1)r)})^{(r)}
\rightarrow
\St_r\otimes(\St_{(n-1)r}\otimes S)^{(r)},
\]
and it has an $(G,S^{(nr)})$-linear splitting by the induction assumption.
On the other hand, $\id\otimes F^r:\St_{nr}\otimes S^{(r)}
\rightarrow \St_{nr}\otimes S$ splits by assumption, as
$\St_{nr}\cong \St_r\otimes \St_{(n-1)r}^{(r)}$.
Thus the composite
\[
\St_{nr}\otimes S^{(nr)}\xrightarrow{\id\otimes (F^{(n-1)r})^{(r)} }
\St_{nr}\otimes S^{(r)}\xrightarrow{\id\otimes F^r}
\St_{nr}\otimes S,
\]
which agrees with $\id\otimes F^{nr}$, has a splitting, as desired.
\end{proof}

\begin{lemma}\label{f-reg-derived.thm}
Let $S=\bigoplus_{n\geq 0}S_n$ be a finitely generated
positively graded $G$-algebra which is an integral domain.
Then the following are equivalent.
\begin{description}
\item[1] $S$ is $G$-strongly $F$-regular.
\item[2] $S$ is $[G,G]$-strongly $F$-regular.
\item[3] $S$ is $\Gamma$-strongly $F$-regular, where $\Gamma\rightarrow[G,G]
$ is the universal covering.
\end{description}
\end{lemma}

\begin{proof}
The implications {\bf 1$\Rightarrow$2$\Leftrightarrow$3} is trivial.
We prove the direction {\bf 3$\Rightarrow$1}.
Replacing $G$ by $\tilde G$ if necessary, we may assume that 
$G=R\times \Gamma$, where $R$ is a torus, and $\Gamma$ is a semisimple
and simply connected algebraic group.
Let $a\in S^G$ be any nonzero homogeneous element.
Then by assumption, the $(R,(S^{(r)})^\Gamma)$-linear map
\[
(aF^r)^*:\Hom_{\Gamma,S^{(r)}}(\St_r\otimes S,
\St_r\otimes S^{(r)})\rightarrow
\Hom_{\Gamma,S^{(r)}}(\St_r\otimes S^{(r)},\St_r\otimes S^{(r)})
\]
is surjective.
Taking the $R$-invariant, 
\[
(aF^r)^*:\Hom_{G,S^{(r)}}(\St_r\otimes S,\St_r\otimes S^{(r)})\rightarrow
\Hom_{G,S^{(r)}}(\St_r\otimes S^{(r)},\St_r\otimes S^{(r)})
\]
is still surjective, since $R$ is linearly reductive.
This is what we wanted to prove.
\end{proof}

The following is proved similarly.

\begin{lemma}\label{f-pure-derived.thm}
Let $S=\bigoplus_{n\geq 0}S_n$ be a finitely generated
$G$-algebra.
Then the following are equivalent.
\begin{description}
\item[1] $S$ is $G$-$F$-pure.
\item[2] $S$ is $[G,G]$-$F$-pure.
\item[3] $S$ is $\Gamma$-$F$-pure, where $\Gamma\rightarrow[G,G]$ is the
universal covering.
\qed
\end{description}
\end{lemma}

\begin{lemma}\label{G-F-pure.thm}
Let $G$ be semisimple and simply connected.
Then $k[G]^U$ is $G$-$F$-pure.
\end{lemma}

\begin{proof} This is \cite[Lemma~3]{Hashimoto2}.
\end{proof}

The following is the main theorem of this paper.

\begin{theorem}\label{main.thm}
Let $S=\bigoplus_{n\geq 0}S_n$ be a finitely generated
positively graded $G$-algebra.
Assume that
\begin{description}
\item[1] $S$ is $F$-rational and Gorenstein.
\item[2] $S$ is $G$-$F$-pure.
\end{description}
Then $S$ is a $G$-strongly $F$-regular integral domain.
\end{theorem}

\begin{proof}
Note that $S$ is normal \cite[(4.2)]{HH3}.
As $S$ is positively graded, $S$ is an integral domain.

Replacing $G$ by $\Gamma$, where $\Gamma\rightarrow [G,G]$ is the
universal covering, we may assume that $G$ is semisimple 
and
simply connected,
by Lemma~\ref{f-reg-derived.thm} and Lemma~\ref{f-pure-derived.thm}.

As $S$ is $G$-$F$-pure, there exists some $l\geq 1$ such that
$\id\otimes F^l:\St_l\otimes S^{(l)}\rightarrow \St_l\otimes S$ has
a $(G,S^{(l)})$-linear  splitting
$\psi:\St_l\otimes S\rightarrow \St_l\otimes S^{(l)}$.

Note that any graded
$(G,S)$-module which is rank one free as an $S$-module
is of the form $S(n)$, where $S(n)$ is $S$ as a $(G,S)$-module, but
the grading is given by $S(n)_i=S_{n+i}$.
In fact, let $-n$ be the generating degree of the rank one free graded
$(G,S)$-module, say $M$, 
then $M_{-n}\otimes S\rightarrow M$ is a $(G,S)$-isomorphism.
As $M_{-n}$ is trivial as a $G$-module (since $G$ is semisimple),
$M_{-n}\cong k(n)$ as a graded $G$-module.
Thus $M\cong k(n)\otimes S\cong S(n)$.

Let $a$ be the $a$-invariant of the Gorenstein positively graded ring $S$.
Namely, $\omega_S\cong S(a)$ (as a 
graded $(G,S)$-module, see the last paragraph).
Then
\[
\Hom_{S^{(r)}}(S,S^{(r)})\cong
\Hom_{S^{(r)}}(S,(\omega_S)^{(r)}(-p^ra))\cong \omega_S(-p^ra)\cong
S((1-p^r)a)
\]
for $r\geq 0$.

Let $\sigma$ be any nonzero element of $\Hom_{S^{(1)}}(S,S^{(1)})_{(p-1)a}
\cong S_0$.
As $S_0=k$ is $G$-trivial,
$\sigma:S\rightarrow S^{(1)}$ is $(G,S^{(1)})$-linear of degree $(p-1)a$.

For $r\geq 0$, let $\sigma_r$ be the composite
\[
S\xrightarrow{\sigma}S^{(1)}\xrightarrow{\sigma^{(1)}}S^{(2)}
\xrightarrow{\sigma^{(2)}}\cdots
\xrightarrow{\sigma^{(r-1)}}S^{(r)}.
\]
It 
is $(G,S^{(r)})$-linear of degree $(p^r-1)a$.
Note that $\sigma_u=\sigma_{u-r}^{(r)}\sigma_r$ for $u\geq r$.

Hence by the composite map
\begin{multline}\label{Q.eq}
Q_{r,u}: \Hom_{S^{(r)}}(S,S^{(r)})
\xrightarrow{\sigma_{u-r}^{(r)}}
\Hom_{S^{(r)}}(S,\Hom_{S^{(u)}}(S^{(r)},S^{(u)}))\\
\cong
\Hom_{S^{(u)}}(S^{(r)}\otimes_{S^{(r)}}S,S^{(u)})
\cong
\Hom_{S^{(u)}}(S,S^{(u)}),
\end{multline}
the element $\sigma_r$ is mapped to $\sigma_u$,
where the first map $\sigma^{(r)}_{u-r}$ maps $f\in
\Hom_{S^{(r)}}(S,S^{(r)})$ to the map $x\mapsto f(x)\cdot \sigma^{(r)}_
{u-r}$.
More precisely, we have $Q_{r,u}(f)= \sigma_{u-r}^{(r)}\circ f$.

By the induction on $u$, it is easy to see that $Q_{1,u}$ is an isomorphism,
and $\sigma_u$ is a generator of
the rank one $S$-free module $\Hom_{S^{(u)}}(S,S^{(u)})$.
It follows that $Q_{r,u}$ is an isomorphism for any $u\geq r$.

We continue the proof of Theorem~\ref{main.thm}.
Take a nonzero homogeneous element $b$ of $A=S^G$.
It suffices to show that there exists some $u\geq 1$ such that
$\id\otimes 
bF^u:\St_u\otimes S^{(u)}\rightarrow \St_u\otimes S$ 
splits as a $(G,S^{(u)})$-linear map.

As $S$ is $F$-rational Gorenstein, it is strongly $F$-regular
by Lemma~\ref{F-sing.thm}, {\bf (viii)}.
So there exists some $r\geq 1$ such that
\[
(bF^r_S)^*:
\Hom_{S^{(r)}}(S,S^{(r)})\rightarrow \Hom_{S^{(r)}}(S^{(r)},S^{(r)})
\]
given by $(bF^r_S)^*(\varphi)=\varphi b F^r_S$ is surjective.

Let $V$ be the degree $-(p^r-1)a-d$ component of $S$,
where $d$ is the degree of $b$.
Note that $V\cong \Hom_{S^{(r)}}(S,S^{(r)})_{-d}$ is mapped onto 
$k\cong \Hom_{S^{(r)}}(S^{(r)},S^{(r)})_0$ by $(bF^r_S)^*$.
In particular, $-(p^r-1)a-d\geq 0$.
So $a\leq 0$.
If $S\neq k$, then it is easy to see that $a<0$.

By Lemma~\ref{steinberg.thm}, there exists some $u_0\geq 1$ such that
for any $u\geq u_0$,
for any subquotient $W$ of $V$, and any $G$-linear 
nonzero map $f:\St_u\otimes W
\rightarrow \St_u$, there exists some $G$-linear map $g:\St_u
\rightarrow \St_u\otimes W$ such that $fg=\id$.
Take $u\geq u_0$ such that $u-r$ is divisible by $l$.

Now the diagram
\[
\xymatrix{
\Hom_{S^{(r)}}(S,S^{(r)}) \ar@{>>}[r]^{(bF^r)^*} \ar[d]^{Q_{r,u}}_{\cong} &
\Hom_{S^{(r)}}(S^{(r)},S^{(r)}) \ar[d]^{Q_{0,u-r}^{(r)}}_{\cong} \\
\Hom_{S^{(u)}}(S,S^{(u)}) \ar[r]^{(bF^r)^*} &
\Hom_{S^{(u)}}(S^{(r)},S^{(u)})
}
\]
is commutative.
So the bottom $(bF^r)^*$ is surjective.
Let us consider the surjection
\[
(bF^r)^*: W=((bF^r)^*)^{-1}(k\cdot \sigma_{u-r}^{(r)})
\cap
\Hom_{S^{(u)}}(S,S^{(u)})_{ap^r(p^{u-r}-1)-d}
\rightarrow 
k\cdot \sigma_{u-r}^{(r)}.
\]
By definition, 
$W$ is
contained in the degree $ap^r(p^{u-r}-1)-d$ component
of $\Hom_{S^{(u)}}(S,S^{(u)})=S(-a(p^u-1))$, which is isomorphic to $V$
as a $G$-module.
So $W$ is isomorphic to a $G$-submodule of $V$.

Let $E:=\Hom(\St_u,\St_u)$.
Then by the choice of $u_0$ and $u$, there exists some $G$-linear map 
$g_1:E\rightarrow E\otimes W$
such that the composite
\[
E\xrightarrow{g_1} E\otimes W\xrightarrow{1\otimes (bF^r)^*} 
E\otimes (k\cdot \sigma_{u-r}^{(r)})
\]
maps $\varphi$ to $\varphi\otimes \sigma_{u-r}^{(r)}$.

We identify $E\otimes\Hom_{S^{(u)}}(S,S^{(u)})$
by $\Hom_{S^{(u)}}(\St_u\otimes S,\St_u\otimes S^{(u)})$ in a natural way.
Similarly, $E\otimes\Hom_{S^{(r)}}(S^{(r)},S)$ is identified with
$\Hom_{S^{(r)}}(\St_u\otimes S^{(r)},\St_u\otimes S)$, and so on.

Then letting $\nu:=g_1(\id_{\St_u})$, the composite
\[
\St_u\otimes S^{(r)}\xrightarrow{\id\otimes
bF^r}\St_u\otimes S\xrightarrow{\nu}
\St_u\otimes S^{(u)}
\]
agrees with $\id\otimes \sigma_{u-r}^{(r)}$.

Since $S$ is $G$-$F$-pure, $u-r$ is a multiple of $l$, and 
$\St_u\cong \St_r\otimes \St_{u-r}^{(r)}$, 
there exists some $(G,S^{(u)})$-linear map $\Phi:\St_u\otimes
S^{(r)}\rightarrow \St_u\otimes S^{(u)}$ such that
$\Phi\circ(\id_{\St_u}\otimes F_S^{u-r})=\id$
by Lemma~\ref{multiple.thm}.

Viewing $\Phi$ as an element of $E\otimes \Hom_{S^{(u)}}(S^{(r)},S^{(u)})$,
let $$\beta\in E\otimes \Hom_{S^{(r)}}(S^{(r)},S^{(r)})$$ be the element
$(\id_E\otimes (Q_{r,u}^{(r)})^{-1})(\Phi)$.
In other words, $\beta:\St_u\otimes S^{(r)}\rightarrow \St_u\otimes S^{(r)}$ 
is the unique map such that the composite
\[
\St_u\otimes S^{(r)}\xrightarrow \beta
\St_u\otimes S^{(r)}\xrightarrow{\id\otimes \sigma_{u-r}^{(r)}}
\St_u\otimes S^{(u)}
\]
is $\Phi$.

Write $\beta=\sum_i \varphi_i\otimes a_i^{(r)}$, where $\varphi_i\in E$ and
$a_i\in S$.
Define $\beta'
\in E\otimes \Hom_S(S,S)$ by $\beta'=\sum_i\varphi_i\otimes a_i^{p^r}$.
Then it is easy to check that $(\id\otimes bF^r)\circ \beta
=\beta'\circ (\id\otimes
bF^r)$ as maps $\St_u\otimes S^{(r)}\rightarrow \St_u\otimes S$.

Combining the observations above,
the whole diagram of $(G,S^{(u)})$-modules 
\[
\xymatrix{
~\St_u\otimes S^{(u)}~ \ar[r]^{\id\otimes F^{u-r}} \ar[d]^{\id} &
~\St_u\otimes S^{(r)}~ \ar[ld]_{\Phi} \ar[d]^\beta 
\ar[dr]^{\id\otimes bF^r} & \\
~\St_u\otimes S^{(u)}~ &
~\St_u\otimes S^{(r)}~ \ar[l]_{\id\otimes \sigma_{u-r}^{(r)}} 
\ar[d]^{\id\otimes bF^r} &
\St_u\otimes S \ar[dl]^{\beta'} \\
 & \St_u\otimes S \ar[ul]^\nu & 
}
\]
is commutative.
So $\id\otimes bF^u:\St_u\otimes S^{(u)}\rightarrow \St_u\otimes S$
has a $(G,S^{(u)})$-linear splitting $\nu \beta'$.
\end{proof}

\begin{corollary}\label{main-cor1.thm}
Let $S$ be as in {\rm Theorem~\ref{main.thm}}.
Then $S^U$ is finitely generated and strongly $F$-regular.
\end{corollary}

\begin{proof}
Finite generation is by \cite[Theorem~9]{Grosshans2}.

We prove the strong $F$-regularity.
We may assume that $G$ is semisimple and simply connected.
Then $k[G]^U$ is a strongly $F$-regular Gorenstein domain by
Lemma~\ref{k[G]^U.thm}.
Hence the tensor product $S\otimes k[G]^U$ is also a strongly $F$-regular
Gorenstein domain, see \cite[Theorem~5.2]{Hashimoto4}.
As $S$ is assumed to be $G$-$F$-pure and $k[G]^U$ is $G$-$F$-pure
by Lemma~\ref{G-F-pure.thm}, the tensor product $S\otimes k[G]^U$ is
also $G$-$F$-pure by Lemma~\ref{tensor-G-F-pure.thm}.
Hence by the theorem, $S\otimes k[G]^U$ is $G$-strongly $F$-regular.
It follows that $(S\otimes k[G]^U)^G$ is strongly $F$-regular
by Lemma~\ref{gsfr-sfr.thm}.
As $S^U\cong (S\otimes k[G]^U)^G$ (see the proof of \cite[(1.2)]{Grosshans}.
See also \cite[Lemma~4.1]{Dolgachev}),
we are done.
\end{proof}

\begin{corollary}\label{w-main.thm}
Let $V$ be a finite dimensional $G$-module, and assume that
$S=\Sym V$ has a good filtration as a $G$-module.
Then $S^U$ is finitely generated and strongly $F$-regular.
\end{corollary}

\begin{proof}
Follows immediately from Corollary~\ref{main-cor1.thm} and
Theorem~\ref{good-filtration.thm}.
\end{proof}

\section{The unipotent radicals of parabolic subgroups}

Let the notation be as in the introduction.
Let $I$ be a subset of $\Delta$.
Let $L=L_I$ be the corresponding Levi subgroup $C_G(\bigcap_{\alpha\in I}
(\Ker \alpha)^\circ)$, where $(?)^\circ$ denotes the identity component,
and $C_G$ denotes the centralizer.
Let $P=P_I$ be the parabolic subgroup generated by $B$ and $L$.
Let $U_P$ be the unipotent radical of $P$.
Let $B_L:=B\cap L$, and $U_L$ the unipotent radical of $B_L$.

Here are two theorems due to Donkin.

\begin{theorem}[Donkin {\cite[(1.2)]{Donkin3}}]\label{Donkin.thm}
Let $w_0$ and $w_L$ denote the longest elements of the Weyl groups of 
$G$ and $L$, respectively.
For $\lambda\in X^+$, we have $\nabla_G(\lambda)^{U_P}\cong\nabla_L(
w_Lw_0\lambda)$ as $L$-modules.
\end{theorem}

\begin{theorem}[Donkin {\cite[(1.4)]{Donkin3}, \cite[(3.9)]{Donkin4}}]%
\label{Donkin2.thm}
Let
\[
0\rightarrow M_1\rightarrow M_2\rightarrow M_3\rightarrow 0
\]
be a short exact sequence of $G$-modules.
If $M_1$ is good, then
\[
0\rightarrow M_1^{U_P}\rightarrow M_2^{U_P}\rightarrow M_3^{U_P}\rightarrow 0
\]
is exact.
In other words, if $M$ is a good $G$-module, then
$R^i(H^0({U_P},?)\circ \res_{U_P}^G)(M)=0$ for $i>0$.
\end{theorem}

From these two theorems, it follows immediately:

\begin{lemma}\label{U_P-good.thm}
Let $M$ be a good $G$-module.
Then $M^{U_P}$ is a good $L$-module.
\end{lemma}

So we have:

\begin{proposition}\label{U_P-main.thm}
Let $k$ be of positive characteristic.
Let $S$ be a finitely generated positively graded $G$-algebra.
Assume that $S$ is Gorenstein $F$-rational, and $G$-$F$-pure.
Then $S^{U_P}$ is finitely generated and $F$-rational.
\end{proposition}

\begin{proof}
By Lemma~\ref{U_P-good.thm}, $S^{U_P}$ is good as an $L$-module.
By Corollary~\ref{main-cor1.thm},
$(S^{U_P})^{U_L}\cong S^U$ is finitely generated and strongly $F$-regular.
By Corollary~\ref{U-F-rational.thm}, applied to the action of 
$L$ on $S^{U_P}$, we have that $S^{U_P}$ is finitely generated and
$F$-rational.
\end{proof}

\begin{corollary}\label{w-U_P.thm}
Let $k$ be of positive characteristic.
Let $V$ be a finite dimensional $G$-module, and assume that
$S=\Sym V$ is good as a $G$-module.
Then $S^{U_P}$ is a finitely generated strongly $F$-regular Gorenstein UFD.
\end{corollary}

\begin{proof}
As $U_P$ is unipotent, $S^{U_P}$ is a UFD by
Remark~3 after Proposition~2 of \cite{Popov3}.

On the other hand, $S$ satisfies the assumption of 
Proposition~\ref{U_P-main.thm} by Theorem~\ref{good-filtration.thm}.
So by
Proposition~\ref{U_P-main.thm}, $S^{U_P}$ is finitely generated and
$F$-rational.
Being a finitely generated Cohen--Macaulay UFD, it is Gorenstein
\cite{Murthy}, 
and hence is strongly $F$-regular 
by Lemma~\ref{F-sing.thm}, {\bf (viii)}.
\end{proof}

\begin{remark}\label{zero.rem}
Let $k$ be of characteristic zero.
The characteristic-zero counterpart of Proposition~\ref{U_P-main.thm}
is stated as follows:
If $S$ is a finitely generated $G$-algebra with rational singularities,
then $S^{U_P}$ is finitely generated with rational singularities.
This is proved in the same line as Proposition~\ref{U_P-main.thm}.
Note that $S^U$ is finitely generated with rational singularities by
\cite[Corollary~4, Theorem~6]{Popov2}.
Then applying \cite[Corollary~4, Theorem~6]{Popov2} again to
the action of $L$ on $S^{U_P}$, $S^{U_P}$ is finitely generated and
has rational singularities, since $(S^{U_P})^{U_L}\cong S^U$ is so.

The characteristic-zero counterpart of Corollary~\ref{w-U_P.thm}
is stated as follows:
If $S$ is a finitely generated $G$-algebra with rational singularities
and is a UFD, then $S^{U_P}$ is a Gorenstein finitely generated 
UFD which is of strongly $F$-regular type.
As we have already seen, $S^{U_P}$ is finitely generated with
rational singularities.
$S^{U_P}$ is a UFD by Remark~3 after Proposition~2 of \cite{Popov3}.
As $S^{U_P}$ is also Cohen--Macaulay \cite[p.~50, Proposition]{TE}, 
$S^{U_P}$ is
Gorenstein \cite{Murthy}.
A Gorenstein finitely generated algebra with 
rational singularities is of strongly $F$-regular type, 
see \cite[(1.1), (5.2)]{Hara}.
\end{remark}

\section{Applications}

The following is pointed out in the proof of 
\cite[(5.2.3)]{SvdB}.

\begin{lemma}\label{SvdB.thm}
Let $K$ be a field of characteristic zero, 
$H$ be an extension of a finite group scheme by a torus over $K$, 
and $A$ a finitely generated $H$-algebra
of strongly $F$-regular type.
Then $A^H$ is of strongly $F$-regular type.
\end{lemma}

\begin{proof} Set $B=A^H$.
Let $\bar K$ be the algebraic closure of $K$.
As can be seen easily, if $\bar K\otimes_K B$ is of strongly $F$-regular
type, then so is $B$.
Since $\bar K\otimes_K B\cong (\bar K\otimes_K A)^{\bar K\otimes_K H}$,
replacing $K$ by $\bar K$, we may assume that $K$ is algebraically closed.
Then $H$ is an extension of a finite group $\Gamma$ by a split torus
$\Bbb G_m^r$ for some $r$.
As $A^H\cong (A^{\Bbb G_m^r})^\Gamma$, we may assume that $H$ is either a
split torus $\Bbb G_m^r$ or a finite group $\Gamma$.

Now we can take a finitely generated $\Bbb Z$-subalgebra $R$ of $K$
and a finitely generated flat $R$-algebra $A_R$
such that $K\otimes_R A_R\cong A$, and for any closed point $x$ of
$\Spec R$, $\kappa(x)\otimes_R A_R$ is strongly $F$-regular.
Extending $R$ if necessary, we have an action of $H_R$ on $A_R$ which is
extended to the action of $H$ on $A$, where $H_R=(\Bbb G_m)_R^r$ or $H_R=
\Gamma$.
Extending $R$ if necessary, we may assume that $n\in R^\times$, where 
$n$ is the order of $\Gamma$, if $H=\Gamma$.

Now set $B_R:=A_R^{H_R}$.

If $H=(\Bbb G_m)_R^r$, then $B_R$ is the degree zero component of the
$\Bbb Z^r$-graded finitely generated $R$-algebra $A_R$, and it is
finitely generated, and is a direct summand subring of $A_R$.

If $H=\Gamma$, then $B_R\rightarrow A_R$ is an integral extension
and $B_R$ is finitely generated by \cite[(7.8)]{AM}.
As $\rho:A_R\rightarrow B_R$ given by $\rho(a)=(1/n)\sum_{\gamma\in\Gamma}
\gamma a$ is a splitting, $B_R$ is a direct summand subring of $A_R$.

In either case, $B_R$ is finitely generated over $R$, so extending
$R$ if necessary, we may assume that $B_R$ is $R$-flat.
Note that $B\cong K\otimes_R B_R$, since $K$ is $R$-flat, and the
invariance is compatible with a flat base change.
Note also that $\kappa(x)\otimes_R B_R$ is a direct summand subring
of $\kappa(x)\otimes_R A_R$, and $\kappa(x)\otimes_R A_R$ is 
strongly $F$-regular.
Hence $\kappa(x)\otimes_R B_R$ is strongly $F$-regular by 
Lemma~\ref{F-sing.thm}, {\bf (iii)}.
This shows that $A^H=B$ is of strongly $F$-regular type.
\end{proof}

The following is a refinement of \cite[(5.2.3)]{SvdB}.

\begin{corollary}
Let $K$ be a field of characteristic zero, $H$ an affine algebraic group
scheme over $K$ such that $H^\circ$ is reductive.
Let $S$ be a finitely generated $H$-algebra 
which has rational singularities and is a UFD.
Then $S^H$ is of strongly $F$-regular type.
\end{corollary}

\begin{proof}
Let $H':=[H^\circ,H^\circ]$.
Then $\bar K\otimes_K H'$ is semisimple, and does not have
a nontrivial character.
Thus $S^{H'}$ has rational singularities by Boutot's theorem \cite{Boutot}
and is a UFD by \cite[(4.28)]{UFD}.
So it is of strongly $F$-regular type by Hara \cite[(1.1), (5.2)]{Hara}.
As $(H/H')^\circ$ is a torus, $S^H=(S^{H'})^{H/H'}$ is of strongly $F$-regular
type by Lemma~\ref{SvdB.thm}.
\end{proof}

\begin{theorem}\label{quiver.thm}
Let $k$ be an algebraically closed field, and
$Q=(Q_0,Q_1,s,t)$ a finite quiver,
where $Q_0$ is the set of vertices, $Q_1$ is the set of arrows, and
$s$ and $t$ are the source and the target maps $Q_1\rightarrow Q_0$,
respectively.
Let $d:Q_0\rightarrow \Bbb N$ be a map.
For $i\in Q_0$, 
set $M_i:=k^{d(i)}$, and let $H_i$ be any closed subgroup scheme
of $\GL(M_i)$ of the following:
\begin{description}
\item[(1)] $\GL(M_i)$, $\SL(M_i)$;
\item[(2)] $\Sp_{d(i)}$ \(in this case, $d(i)$ is required to be even\);
\item[(3)] $\SO_{d(i)}$ \(in this case, the characteristic of $k$ must
not be two\);
\item[(4)] Levi subgroup of any of {\bf(1)}--{\bf(3)};
\item[(5)] Derived subgroup of any of {\bf(1)}--{\bf(4)};
\item[(6)] Unipotent radical of a parabolic subgroup of 
any of {\bf(1)}--{\bf(5)};
\item[(7)] Any subgroup $H_i$ of $\GL(M_i)$ with a closed
normal subgroup $N_i$
of $H_i$ such that $N_i$ is any of {\bf(1)}--{\bf(6)}, 
and $H_i/N_i$ is a linearly
reductive group scheme.
In characteristic zero, we require that $(H_i/N_i)^\circ$ is a torus.
\end{description}
Set $H:=\prod_{i\in Q_0}H_i$ and $M:=\prod_{\alpha\in Q_1}
\Hom(M_{s(\alpha)},M_{t(\alpha)})$.
Then $(\Sym M^*)^H$ is finitely generated, and strongly $F$-regular 
if the characteristic of $k$ is positive, and strongly $F$-regular type
if the characteristic of $k$ is zero.
\end{theorem}

\begin{proof}
If $H_i$ satisfies {\bf(7)} and the corresponding $N_i$ satisfies
{\bf(x)}, where $1\leq {\bf x}\leq 6$, then we say that $H_i$ is of type
{\bf (7,x)}.

Note that $\Sym M^*\cong \bigotimes_{\alpha\in Q_1}
\Sym(M_{s(\alpha)}\otimes M_{t(\alpha)}^*)$.

First we prove that $\Sym M^*$ has a good filtration as an $H$-module
if each $H_i$ is as in {\bf(1)}--{\bf(5)}.
To verify this, we only have to show that $\Sym(M_{s(\alpha)}\otimes
M_{t(\alpha)}^*)$ is a good $H$-module for each $\alpha$, by Mathieu's
tensor product theorem \cite{Mathieu}.
This module is trivial as an $H_i$-module if $i\neq s(\alpha),t(\alpha)$.
Thus it suffices to show that this is good as an $H_{s(\alpha)}\times
H_{t(\alpha)}$-module if $s(\alpha)\neq t(\alpha)$, and
as an $H_{s(\alpha)}$-module if $s(\alpha)=t(\alpha)$, 
see \cite[Lemma~4]{Hashimoto}.
By \cite[Lemma~3, {\bf 3}, {\bf 5}, {\bf 6}]{Hashimoto} and
\cite[(3.2.7), (3.4.3)]{Donkin}, the assertion is true
for $H_{s(\alpha)}$, $H_{t(\alpha)}$ of {\bf(1)}--{\bf(3)}.
By Mathieu's theorem \cite[Theorem~1]{Mathieu}, the groups of type {\bf(4)}
is also allowed.
By \cite[(3.2.7)]{Donkin} again, the groups of type {\bf (5)} is also 
allowed.
By \cite[Theorem~6]{Hashimoto}, the conclusion of the theorem holds
this case.

We consider the general case.
If $H_i$ is of the form {\bf (1)}--{\bf (5)}, then considering 
$N_i=U_i\subset B_i\subset H_i$, where $B_i$ is a Borel subgroup of $H_i$ and
$U_i$ its unipotent radical, $B_i$ is a group of the form {\bf (7)},
and as the $H_i$-invariant and the $B_i$-invariant are the same thing
for an $H_i$-module, we may replace $H_i$ by $B_i$ without changing 
the invariant subring.
Hence in this case, we may assume that $H_i$ is of the form {\bf (7,6)}.
Clearly, a group of the form {\bf (6)} is also of the form {\bf (7,6)},
letting $N_i=H_i$.
So we may assume that each $H_i$ is of type {\bf (7)}.
If $(\Sym M^*)^N$ is strongly $F$-regular (type), 
where $N=\prod_{i\in Q_0}N_i$,
then $(\Sym M^*)^H\cong ((\Sym M^*)^N)^{H/N}$ is also strongly $F$-regular
in positive characteristic,
since $H/N\cong \prod_{i\in Q_0}H_i/N_i$ is linearly reductive and
$(\Sym M^*)^H$ is a direct summand subring of $(\Sym M^*)^N$.
In characteristic zero, $(H/N)^\circ$ is a torus, and we can invoke
Lemma~\ref{SvdB.thm}.
Thus we may assume that each $H_i$ is of the form {\bf (1)--(6)}.
Then again by the argument above, we may assume that each $H_i$ is of the form
{\bf (7,6)}.
Again by the argument above, we may assume that each $H_i$ is of the form
{\bf (6)}.
Now suppose that $H_i\subset G_i\subset \GL(M_i)$, and 
each $G_i$ if of the form {\bf (1)--(5)}, and $H_i$ is the unipotent radical
of the parabolic subgroup $P_i$ of $G_i$.
Then letting $G:=\prod G_i$ and $P:=\prod P_i$, $H=\prod H_i$ is the
unipotent radical of the parabolic subgroup $P$ of $G$.
As $\Sym M^*$ has a good filtration as an $G$-module by the first paragraph,
$(\Sym M^*)^H$ is finitely generated and strongly $F$-regular (type)
by Corollary~\ref{w-U_P.thm} and Remark~\ref{zero.rem}.
\end{proof}

This covers Example~1 and Example~2 of \cite{Hashimoto},
except that we do not consider the case $p=2$ here, if $\Orth_n$ or $\SO_n$ is
involved.
For example,

\begin{example}
Let $Q=1\rightarrow 2\rightarrow 3$, $(d(1),d(2),d(3))=
(m,t,n)$, $H_1=H_3=\{e\}$, and $H_2=\GL_t$.
Then $M=\Hom(M_2,M_3)\times \Hom(M_1,M_2)$, and $M\rightarrow M\dq H$
is identified with
\[
\pi: M\rightarrow Y_t=\{f\in\Hom(M_1,M_3)\mid\rank f\leq t\},
\]
where $\pi(\varphi,\psi)=\varphi\psi$ (De Concini--Procesi 
\cite{DP}).
Thus (the coordinate ring of) $Y_t$ is strongly $F$-regular (type), as was
proved by Hochster--Huneke \cite[(7.14)]{HH4} ($F$-regularity and strong 
$F$-regularity are equivalent for positively graded rings, see 
Lemma~\ref{F-sing.thm}).
\end{example}

Next we consider an example which really requires a 
group of type {\bf (7)} in
Theorem~\ref{quiver.thm}.

Let $K$ be a field, and $M=K^m$, $N=K^n$.
Let $1\leq s\leq n$, and $\ua=(0=a_0<a_1<\cdots<a_s=n)$ be an increasing
sequence of integers.
Let $\fG$, $\fS$, and $\fT$ be disjoint subsets of $\{1,\ldots,s\}$
such that $\fG\coprod\fS\coprod\fT=\{1,\ldots,s\}$.
Let 
\[
H=H(\ua;\fG,\fS,\fT):=
\left(\,
\begin{array}{cccc}
\hline
\multicolumn{1}{|c}{H_1} &
\multicolumn{3}{c|}{}
\\
\cline{1-1}
 & \multicolumn{1}{|c}{H_2} &
\multicolumn{2}{c|}{~~~*}
\\
\cline{2-2}
 & & \multicolumn{1}{|c}{\ddots} &
\multicolumn{1}{c|}{}
\\
\cline{3-3}
 & O~~ & & \multicolumn{1}{|c|}{H_s} \\
\cline{4-4}
\end{array}
\,\right)
\subset\GL_m(K)\cong \GL(M),
\]
where $H_l$ is $\GL_{a_l-a_{l-1}}$ if $l\in \fG$, 
$\SL_{a_l-a_{l-1}}$ if $l\in \fS$, and
$\{E_{a_l-a_{l-1}}\}$ if $l\in \fT$.

Let us consider the symmetric algebra $S=\Sym (M\otimes N)$.
It is a graded polynomial algebra over $K$ with each variable degree one.
Let $e_1,\ldots,e_m$ and $f_1,\ldots,f_n$ be the 
standard bases of $M=K^m$ and $N=K^n$, respectively.
For sequences $1\leq c_1,\ldots,c_u\leq m$ and $1\leq d_1,\ldots,d_u\leq n$,
we define $[c_1,\ldots,c_u\mid d_1,\ldots,d_u]$ to be the determinant
$\det(e_{c_i}\otimes f_{d_j})_{1\leq i,j\leq u}$.
It is a minor of the matrix $(e_i\otimes f_j)$ up to sign, or zero.
Let $\Sigma$ be the set of minors
\begin{multline*}
\{[c_1,\ldots,c_u\mid d_1,\ldots,d_u]\mid
1\leq u\leq \min(m,n),\\
1\leq c_1<\cdots<c_u\leq m,\;
1\leq d_1<\cdots<d_u\leq n
\}.
\end{multline*}
We say that $[c_1,\ldots,c_u\mid d_1,\ldots,d_u]\leq
[c'_1,\ldots,c'_v\mid d'_1,\ldots,d'_v]$
if $u\geq v$, and $c'_i\geq c_i$ and $d'_i\geq d_i$ for $1\leq i\leq v$.
It is easy to see that $\Sigma$ is a distributive lattice.

Set $\epsilon:=\min \fG$.
For $1\leq l<\epsilon$, set
\[
\Gamma_l:=\{[1,\ldots,a_l\mid d_1,\ldots,d_{a_l}]\mid 1\leq d_1<\cdots<d_{a_l}
\leq n\}
\]
if $l\in\fS$, and
\begin{multline*}
\Gamma_l:=\{[c_1,\ldots,c_u\mid d_1,\ldots,d_u]\mid
a_{l-1}<u\leq a_l,\\
1\leq c_1<\cdots<c_u\leq a_l,\;
c_t=t\;(t\leq a_{l-1}),\;
1\leq d_1<\cdots<d_u\leq n\}
\end{multline*}
if $l\in \fT$.
Set $\Gamma=\bigcup_{l<\epsilon}\Gamma_l$.
Note that $\Gamma$ is a sublattice of $\Sigma$.

It is well-known that $S$ is an ASL on $\Sigma$ over $K$ \cite[(7.2.7)]{BH}.
For the definition of ASL, see \cite[(7.1)]{BH}.

\begin{lemma}
Let $B$ be a graded ASL on a poset $\Omega$ over a field $K$.
Let $\Xi$ be a subset of $\Omega$ such that for any two incomparable
elements $\xi,\eta\in\Xi$,
\begin{equation}\label{straightening.eq}
\xi\eta=\sum c_i m_i
\end{equation}
in $S$ with each $m_i$ in the right hand side being a monomial of 
$\Xi$ divisible by an element $\xi_i$ in $\Xi$ smaller than both $\xi$
and $\eta$.
Then the subalgebra $K[\Xi]$ of $B$ is a graded ASL on $\Xi$.
\end{lemma}

\begin{proof}
We may assume that $m_i$ in the right hand side of 
(\ref{straightening.eq}) has the same degree as that of $\xi\eta$ for
each $\xi$, $\eta$, and $m_i$.
For a monomial $m=\prod_{\omega\in\Omega}\omega^{c(\omega)}$,
the weight $w(m)$ 
of $m$ is defined to be $\sum_\omega c(\omega)3^{\coht(\omega)}$,
where $\coht(\omega)$ is the maximum of the lengths of chains 
$\omega=\omega_0<\omega_1<\cdots$ in $\Omega$.
Then $w(mm')=w(m)+w(m')$, and for each $i$, $w(m_i)>w(\xi\eta)$ in 
(\ref{straightening.eq}).
So each time we use (\ref{straightening.eq}) to rewrite a monomial, 
the weight goes up.
On the other hand, there are only finitely many monomials of a given degree,
this rewriting procedure will stop eventually, and we get a linear combination
of standard monomials in $\Xi$.
Now $(H_2)$ condition in \cite[(7.1)]{BH} is clear, while
$(H_0)$ and $(H_1)$ are trivial.
\end{proof}

We call $K[\Xi]$ a {\em subASL} of $B$ generated by $\Xi$ 
if the assumption of the lemma is satisfied.

\begin{theorem}
Let the notation be as above.
Let $H$ act on $S$ via $h(m\otimes n)=h(m)\otimes n$.
Set $A:=S^H$.
Then
\begin{description}
\item[(1)] $A=K[\Gamma]$.
\item[(2)] $K[\Gamma]$ is a subASL of $S=K[\Sigma]$ generated by $\Gamma$.
\item[(3)] $A$ is a Gorenstein UFD.
It is strongly $F$-regular if the characteristic of $K$ is positive, and
is of strongly $F$-regular type if the characteristic of $K$ is zero.
\end{description}
\end{theorem}

\begin{proof}
First we prove that $A$ is strongly $F$-regular (type).
To do so, we may assume that $K=k$ is algebraically closed.
Let $B^+$ be the subgroup of upper triangular matrices in $\GL_m$,
and set $B_H^+:=B^+\cap H$.
Then it is easy to see that $A=S^{B_H^+}$.

Now let $Q$ be the quiver $1\rightarrow 2$, $d=(d(1),d(2))=(m,n)$, 
$G_1=B_H^+\subset \GL_m$, and $G_2=\{e\}$.
Let $U_H^+$ be the unipotent radical of $B_H^+$.
Then $U_H^+$ is the unipotent radical of an appropriate parabolic subgroup
of $\GL_m$, $U_H^+$ is normal in $B_H^+$, and $B_H^+/U_H^+$ is a torus.
Thus the assumption {\bf(7)} of Theorem~\ref{quiver.thm} is satisfied,
and thus $A=S^{B_H^+}$ is strongly $F$-regular (type).

The assertion {\bf(2)} is a consequence of the straightening relation
of the ASL $S$.
See \cite{ABW} for details.

Assume that {\bf (1)} is proved.
Then by the definition of $\Gamma$, letting $M'$ be the subspace of
$M$ spanned by $e_1,\ldots,e_{a_{\epsilon-1}}$, 
$A=\Sym(M'\otimes N)^{H'}$, where $H'=H\cap \GL(M')$,
by {\bf (1)} again
($\GL(M')$ is 
viewed as a subgroup of $\GL(M)$ via $g'(e_i)=e_i$ for $i>a_{\epsilon-1}$).
As $H'$ is connected and 
$\bar K\otimes_K H'$ does not have a nontrivial character, $A$ is a UFD by
\cite[(4.28)]{UFD}, where $\bar K$ is the algebraic closure of $K$.
So assuming {\bf (1)}, the assertion {\bf(3)} is proved.

It remains to prove {\bf (1)}.
It is easy to see that $\Gamma\subset A$.
So it suffices to prove that $\dim_K A_d= \dim_K K[\Gamma]_d$ for each
degree $d\geq 0$.
To do so, we may assume that $K$ is algebraically closed.

Let $P^+$ be the parabolic subgroup $H(\ua;\{1,\ldots,s\},\emptyset,\emptyset)$
of $\GL_m$, and $U_{P^+}$ the unipotent radical of $P^+$.
If 
\[
0\rightarrow M_1\rightarrow M_2\rightarrow M_3\rightarrow 0
\]
is a short exact sequence of good $\GL(M)\times \GL(N)$-modules,
then 
\begin{equation}\label{U_P.eq}
0\rightarrow M_1^{U_{P^+}}\rightarrow 
M_2^{U_{P^+}}\rightarrow M_3^{U_{P^+}}\rightarrow 0
\end{equation}
is an exact sequence of good $P^+/U_{P^+}$-modules by
Lemma~\ref{U_P-good.thm} and Theorem~\ref{Donkin2.thm}.
Note that $P^+/U_{P^+}$ is identified with $\prod_{l=1}^s \GL_{a_l-a_{l-1}}$,
and $H/U_{P^+}$ is identified with its subgroup $\prod_{l=1}^s H_l$.
As each $H_l$ is either $\GL_{a_l-a_{l-1}}$, $\SL_{a_l-a_{l-1}}$, or trivial,
it follows that a good $P^+/U_{P^+}$-module is also good as an 
$H/U_{P^+}$-module.
Applying the invariance functor $(?)^{H/U_{P^+}}$ to (\ref{U_P.eq}), 
\[
0\rightarrow M_1^{H}\rightarrow M_2^{H}\rightarrow M_3^{H}\rightarrow 0
\]
is exact.

Now we employ the standard convention for $\GL(M)$.
Let $T$ be the set of diagonal matrices in $G:=\GL(M)=\GL_m$, 
and we identify $X(T)$ with $\Bbb Z^m$ by the isomorphism
\[
\Bbb Z^m\ni(\lambda_1,\lambda_2,\ldots,\lambda_m)\mapsto
\left(
\begin{pmatrix}
t_1 \\
 & t_2 \\
 & & \ddots \\
 & & & t_m
\end{pmatrix}
\mapsto t^\lambda=t_1^{\lambda_1}t_2^{\lambda_2}\cdots t_m^{\lambda_m}
\right)
\in X(T).
\]
We fix the base of the root system of $\GL(M)$ so that the set of
lower triangular matrices in $\GL(M)$ is negative.
Then the set of dominant weights $X^+_{\GL(M)}$ is the set
\[
\{\lambda=(\lambda_1,\ldots,\lambda_m)\in X(T)\mid \lambda_1\geq\cdots
\geq \lambda_m\}.
\]
We use a similar convention for $\GL(N)$.
See \cite[(II.1.21)]{Jantzen} for more information on this convention.

For $\lambda\in X^+_{\GL(M)}$, $ \nabla_{\GL(M)}(\lambda)^{U_{P^+}}$
is a single dual Weyl module by Theorem~\ref{Donkin.thm}.
But
obviously, the highest weight of $ \nabla_{\GL(M)}(\lambda)^{U_{P^+}}$
is $\lambda$.
Thus $ \nabla_{\GL(M)}(\lambda)^{U_{P^+}}\cong \nabla_{P^+/U_{P^+}}(\lambda)$.
Now the following is easy to verify:

\begin{lemma}\label{sublemma.thm}
For $\lambda=(\lambda_1,\ldots,\lambda_m)\in X^+_{\GL(M)}$,
\[
\nabla_{\GL(M)}(\lambda)^H
\cong
\begin{cases}
\nabla_{\GL_{a_1}}(\lambda(1))\otimes\cdots\otimes\nabla_{\GL_{a_s-a_{s-1}}}
(\lambda(s)) & (\lambda\in\Theta) \\
0 & (\text{otherwise})
\end{cases}
\]
as $P^+/H$-modules, where 
$\lambda(l):=(\lambda_{a_{l-1}+1},\ldots,\lambda_{a_l})$ for each $l$, and
$\Theta$ is the subset of $X^+_{\GL(M)}$ consisting of 
sequences $\lambda=(\lambda_1,\ldots,\lambda_m)$ such that
$\lambda(l)=(0,0,\ldots,0)$ for each $l\in\fG$, and
$\lambda(l)=(t,t,\ldots,t)$ for some $t\in\Bbb Z$ for each $l\in\fS$.
\end{lemma}

Let $r:=\min(m,n)$, and set 
\[
\Cal P(d)=\{\lambda=(\lambda_1,\ldots,\lambda_r)\in\Bbb Z^r\mid
\lambda_1\geq \cdots\geq \lambda_r\geq 0,\; |\lambda|=d\},
\]
where $|\lambda|=\lambda_1+\lambda_2+\cdots+\lambda_r$.
We consider that 
\[
(\lambda_1,\ldots,\lambda_r)=
(\lambda_1,\ldots,\lambda_r,0,\ldots,0), 
\]
and $\Cal P(d)\subset X^+_{\GL(M)}$.
Similarly, we also consider that $\Cal P(d)\subset X^+_{\GL(N)}$.
By the Cauchy formula \cite[(III.1.4)]{ABW}, 
$S_d$ has a good filtration as a $\GL(M)\times\GL(N)$-module 
whose associated graded object is 
\[
\bigoplus_{\lambda\in\Cal P(d)}
\nabla_{\GL(M)}(\lambda)\boxtimes\nabla_{\GL(N)}(\lambda).
\]
Note that $\nabla_{\GL(M)}(\lambda)$ is isomorphic to the Schur module
$L_{\tilde\lambda}M$ in \cite{ABW}, where $\tilde\lambda$ is the 
transpose of $\lambda$.
That is, 
$\tilde\lambda=(\tilde\lambda_1,
\tilde\lambda_2,\ldots)$ is given by 
$\tilde\lambda_i=\#\{j\geq 1 \mid \lambda_j\geq i\}$.

By Lemma~\ref{sublemma.thm}, 
$S_d^H$ has a filtration whose associated graded object is
\[
\bigoplus_{\lambda\in\Cal P(d)\cap\Theta}\nabla_{\GL_{a_1}}(\lambda(1))
\otimes\cdots\otimes\nabla_{\GL_{a_s-a_{s-1}}}(\lambda(s))\boxtimes
\nabla_{\GL(N)}(\lambda).
\]
In particular, 
\begin{equation}\label{dimen.eq}
\dim S_d^H
=
\sum_{\lambda\in\Cal P(d)\cap \Theta}
\dim\nabla_{\GL(N)}(\lambda)\prod_l\dim \nabla_{\GL_{a_l-a_{l-1}}}(\lambda(l)).
\end{equation}

Next we count the dimension of $K[\Gamma]_d$.
This is the number of standard monomials of degree $d$ in $K[\Gamma]$.
For a standard monomial
\[
v=\prod_{b=1}^\alpha [c_{b,1},\ldots,c_{b,\mu_b}\mid 
d_{b,1},\ldots,d_{b,\mu_b}]
\]
(where $[c_{b,1},\ldots,c_{b,\mu_b}\mid d_{b,1},\ldots,d_{b,\mu_b}]$
increses when $b$ increases) in $\Sigma$,
we define $\mu(v)=(\mu_1,\ldots,\mu_\alpha)$, and $\lambda(v)$ its transpose.
Such a standard monomial $v$ of $\Gamma$ of degree $d$ 
exists if and only if 
$\lambda(v)\in\Theta\cap\Cal P(d)$.

For a standard monomial $v$ of $\Sigma$ such that $\lambda(v)=\lambda
\in\Cal P(d)\cap \Theta$,
$v$ is a monomial of $\Gamma$ if and only if the following condition holds.
For each $1\leq b\leq \lambda_1$, $1\leq l\leq s$, and 
each $a_{l-1}<i\leq a_l$, it holds $a_{l-1}<c_{s,i}\leq a_l$.
The number of such monomials agrees with
$\dim\nabla_{\GL(N)}(\lambda)\prod_l\dim\nabla_{\GL_{a_l-a_{l-1}}}
(\lambda(l))$,
as can be seen easily from the standard basis theorem \cite[(II.2.16)]{ABW}.
So $\dim_K K[\Gamma]$ agrees with the right hand side of (\ref{dimen.eq}), 
and we have $\dim_K A_d=\dim_K S_d^H=\dim_K K[\Gamma]_d$, as desired.
\end{proof}

\begin{remark}
The case that $s=2$, $a_1=l$, $\fG=\emptyset$, $\fS=\{2\}$, and
$\fT=\{1\}$ is studied by Goto--Hayasaka--Kurano--Nakamura \cite{GHKN}.
Gorenstein property and factoriality are proved there for this case.
The case that $s=m$, $a_l=l$ ($l=1,\ldots,m$), $\fG=\fS=\emptyset$, 
and $\fT=\{1,\ldots,m\}$ is a very special case of the study of 
Miyazaki \cite{Miyazaki}.
\end{remark}

\section{Openness of good locus}

\paragraph\label{reductive-ring.par}
Let $R$ be a Noetherian commutative ring, and $G$ a split
reductive group over $R$.
We fix a split maximal torus $T$ of $G$ whose embedding
into $G$ is defined over $\Bbb Z$.
We fix a base $\Delta$ of the root system, and let $B$ be the
negative Borel subgroup.
For a dominant weight $\lambda$, the dual Weyl module
$\nabla_G(\lambda)$ is defined to be $\ind_B^G(\lambda)$,
and the Weyl module $\Delta_G(\lambda)$ is defined to be
$\nabla_G(-w_0\lambda)^*$.

A $G$-module $M$ is said to be {\em good} if 
$\Ext^1_G(\Delta_G(\lambda),M)=0$ for any $\lambda
\in X^+$, where $X^+$ is the set of dominant weights,
see \cite[(III.2.3.8)]{Hashimoto5}.

\begin{lemma}
The notion of goodness of a $G$-module $M$ is independent
of the choice of $T$ or $\Delta$, and depends only on
$M$.
\end{lemma}

\begin{proof}
Let $T'$ and $\Delta'$ be another choice of a split
maximal torus defined over $\Bbb Z$ and a base of the
root system (with respect to $T'$).
Let $B'$ be the corresponding negative Borel subgroup.

Assume that $R$ is an algebraically closed field.
Then there exists some $g\in G(R)$ such that
$gBg^{-1}=B'$.
So $\ind_B^G\lambda\cong \ind_{B'}^G(\lambda')$ for any $\lambda\in X(B)$,
where $\lambda'$ is the composite
\[
B'\xrightarrow{b'\mapsto g^{-1}b'g} B\xrightarrow{\lambda}\Bbb G_m.
\]
So this case is clear.

When $R$ is a field, then
a $G$-module $M$ is good if and only if $\bar R\otimes_R M$ is
so as an $\bar R\otimes_R G$-module, and this notion is
independent of the choice of $B$, where $\bar R$ is the
algebraic closure of $R$.

Now consider the general case.
If $M$ is $R$-finite $R$-projective, then the assertion follows
from \cite[(III.4.1.8)]{Hashimoto5} and the discussion above.
If $M$ is general, then $M$ is good if and only if
there exists some filtration
\[
0=M_0\subset M_1\subset M_2\subset \cdots
\]
of $M$ such that $\bigcup_i M_i=M$, and for each
$i\geq 1$, $M_i/M_{i-1}\cong N_i\otimes V_i$ for
some $R$-finite $R$-projective good $G$-module $N_i$
and an $R$-module $V_i$.
Indeed, the only if part is \cite[(III.2.3.8)]{Hashimoto5},
while the if part is a consequence of the goodness of
$N_i\otimes V_i$, see \cite[(III.4.1.8)]{Hashimoto5}.
This notion is independent of the choice of $T$ or $\Delta$,
and we are done.
\end{proof}

Note that if $R\rightarrow R'$ is a Noetherian $R$-algebra,
then an $R'\otimes_R G$-module $M'$ is good if and only
if it is so as a $G$-module.
This comes from the isomorphism
\[
\Ext_G^i(\Delta_G(\lambda),M')\cong \Ext_{R'\otimes G}^i(
\Delta_{R'\otimes G}(\lambda),M').
\]

If
$M$ is a good $G$-module, and $R'$ is $R$-flat or $M$
is $R$-finite $R$-projective, then
$R'\otimes_R M$ is a good $R'\otimes_R G$-module
by \cite[(I.3.6.20)]{Hashimoto5} and
\cite[(III.1.4.8)]{Hashimoto5}, see \cite[(III.2.3.15)]{Hashimoto5}.
If $M$ is good and $V$ is a flat $R$-module,
then $M\otimes V$ is good.
This follows from the canonical isomorphism
\[
\Ext^i_G(\Delta_G(\lambda),M\otimes V)\cong
\Ext^i_G(\Delta_G(\lambda),M)\otimes V,
\]
see \cite[(I.3.6.16)]{Hashimoto5}.

If $R'$ is faithfully flat over $R$ and $R'\otimes_R M$
is good, then $M$ is good by \cite[(I.3.6.20)]{Hashimoto5}.

\paragraph\label{reductive-scheme.par}
Let $S$ be a scheme, and $G$ a reductive group scheme
over $S$, and $X$ a Noetherian $S$-scheme on which $G$
acts trivially.
Let $M$ be a quasi-coherent $(G,\Cal O_X)$-module.
For $(G,\Cal O_X)$-modules, see \cite[Chapter~29]{Hashimoto7}.
Almost by definition, a $(G,\Cal O_X)$-module and a
$(G\times_S X,\Cal O_X)$-module (note that $G\times_S X$ is an $X$-group
scheme) are the same thing.

We say that $M$ is good if there is a Noetherian commutative ring
$R$ and a faithfully flat morphism of finite type
$f:\Spec R\rightarrow X$ such that $G_R:=\Spec R\times_S G$
is a split reductive group scheme over $R$, 
and $\Gamma(\Spec R,f^*M)$ is a good $G_R$-module.
This notion is independent of the choice of $f$ such that
$G_R$ is split reductive.
When $X=\Spec B$ is affine, then we also say that $\Gamma(X,M)$ is
good, if $M$ is good.
If $g:X'\rightarrow X$ is a flat morphism of Noetherian schemes
and $M$ is a good quasi-coherent $(G,\O_X)$-module, then
$g^*M$ is good.
If $M$ is a quasi-coherent $(G,\O_X)$-module, $g$ is faithfully
flat, and $g^*M$ is good, then $M$ is good.

For a quasi-coherent $(G,\Cal O_X)$-module $M$, we define the
good locus of $M$ to be
\[
\Good(M)=\{x\in X\mid \text{$M_x$ is a good 
$(\Spec \O_{X,x}\times_S G)$-module}
\}.
\]
If $g:X'\rightarrow X$ is a flat morphism of Noetherian schemes,
then $g^{-1}(\Good(M))=\Good(g^*M)$.
If $X=\Spec R$ is affine, then for a $(G,R)$-module $N$, $\Good(N)$ stands
for $\Good(\tilde N)$, where $\tilde N$ is the sheaf associated with $N$.

\paragraph\label{split-reductive.par}
Let the notation be as in (\ref{reductive-ring.par}).

For a poset ideal $\pi$ of $X^+$ and a $G$-module $M$, we say that
$M$ belongs to $\pi$ if $M_\lambda=0$ for $\lambda\in X^+\setminus\pi$.

\begin{proposition}\label{belonging.thm}
Let $\pi$ be a poset ideal of $X^+$ and $M$ a $G$-module.
Then the following are equivalent.
\begin{description}
\item[(1)] $M$ belongs to $\pi$.
\item[(2)] For any $R$-finite 
subquotient $N$ of $M$ and any $R$-algebra 
$K$ that is a field, $K\otimes_R N$ belongs to $\pi$.
\item[(3)] For any $R$-finite subquotient $N$ of $M$, any $R$-algebra
$K$ that is a field, and $\lambda\in X^+\setminus \pi$, 
$\Hom_G(\Delta_G(\lambda),K\otimes_R N)=0$.
\item[(4)] For any $\lambda\in X^+\setminus\pi$, 
$\Hom_G(\Delta_G(\lambda),M)=0$.
\item[(5)] $M$ is a $C_\pi$-comodule, where $C_\pi$ is the Donkin
subcoalgebra of $C$ with respect to $\pi$, see 
{\rm\cite[(III.2.3.13)]{Hashimoto5}}.
\end{description}
\end{proposition}

\begin{proof}
{\bf (1)$\Rightarrow$(2)} is obvious.

{\bf (2)$\Rightarrow$(3)} We may assume that $R=K$ and $N=M$.
Then
\begin{multline*}
\Hom_G(\Delta_G(\lambda),M)\cong\Hom_G(M^*,\ind_B^G(-w_0\lambda))\\
\cong\Hom_B(M^*,-w_0\lambda)\cong\Hom_B(w_0\lambda,M)\subset M_{w_0\lambda}=0.
\end{multline*}

{\bf (3)$\Rightarrow$(4)}
As $M$ is the inductive limit of $R$-finite $G$-submodules of $M$, 
we may assume that $M$ is $R$-finite.
We use the Noetherian induction, and we may assume that the implication
is true for $R/I$ for any nonzero ideal $I$ of $R$.
If $R$ is not a domain, then there is a nonzero ideal $I$ of $R$ such that
the annihilator $0:I$ of $R$ is also nonzero.
As $\Hom_G(\Delta_G(\lambda),M/IM)=0$ and $\Hom_G(\Delta_G(\lambda),IM)=0$,
we have that $\Hom_G(\Delta(\lambda),M)=0$.
So we may assume that $R$ is a domain.
Let $N$ be the torsion part of $M$.
Note that 
\[
0\rightarrow N\rightarrow M\rightarrow K\otimes_R M
\]
is exact, where $K$ is the field of fractions of $R$.
Hence $N$ is a $G$-submodule of $M$.
The annihilator of $N$ is nontrivial, and hence $\Hom_G(\Delta_G(\lambda
),N)=0$.
On the other hand, by assumption, 
$\Hom_{G}(\Delta_G(\lambda),K\otimes_R M)=0$.
So $\Hom_G(\Delta_G(\lambda),M)=0$, and we are done.

{\bf (4)$\Rightarrow$(5)} is \cite[(III.2.3.5)]{Hashimoto5}.

{\bf (5)$\Rightarrow$(1)} As the coaction $\omega_M:M\rightarrow M'\otimes_R
C_\pi$ is injective, it suffices to show that $M'\otimes_R C_\pi$ belongs to
$\pi$, where $M'$ is the $R$-module $M$ with the trivial $G$-action.
For this, it suffices to show that $C_\pi$ belongs to $\pi$.
This is proved easily by induction on 
the number of elements of $\pi$, if $\pi$ is finite, almost by the 
definition of the Donkin system \cite[(III.2.2)]{Hashimoto5}, and
the fact that $\nabla_G(\lambda)$ belongs to $\pi$.
Then the general case follows easily from the definition of $C_\pi$, 
see \cite[(III.2.3.13)]{Hashimoto5}.
\end{proof}

\begin{corollary}
Let $M$ be a $G$-module, and $\pi$ a poset ideal of the
set of dominant weights $X^+$.
If $M$ belongs to $\pi$,
then $\Ext^i_G(\Delta_G(\lambda),M)=0$ for $i\geq 0$ and $\lambda\in X^+
\setminus \pi$.
\end{corollary}

\begin{proof}
We use the induction on $i$.
The case $i=0$ is already proved in Proposition~\ref{belonging.thm}.

Let $i>0$.
Let $C_\pi$ denote the Donkin subcoalgebra of $k[G]$.
Consider the exact sequence
\[
0\rightarrow M\xrightarrow{\omega_M} 
M'\otimes_R C_\pi\rightarrow N\rightarrow 0.
\]
Then $N$ belongs to $\pi$, and $\Ext^{i-1}_G(\Delta_G(\lambda),N)=0$
by induction assumption.
On the other hand, as $C_\pi$ is good and $R$-finite $R$-projective 
by construction, $M'\otimes_R C_\pi$ is also good by
\cite[(III.4.1.8)]{Hashimoto5}.
Hence $\Ext_G^i(\Delta_G(\lambda),M'\otimes_R C_\pi)=0$.
By the long exact sequence of the $\Ext$-modules, we have that
$\Ext^i_G(\Delta_G(\lambda),M)=0$.
\end{proof}

\begin{lemma}\label{finite-open.thm}
Let the notation be as in {\rm(\ref{reductive-scheme.par})}.
Let $M$ be a coherent $(G,\O_X)$-module.
Then $\Good(M)$ is Zariski open in $X$.
\end{lemma}

\begin{proof}
Let $f:\Spec R\rightarrow X$ be a faithfully flat morphism of 
finite type such that $G_R$ is split reductive.
Let $M_R:=\Gamma(\Spec R,f^*M)$.
Then $\Good(M_R)=f^{-1}(\Good(M))$.
As $f$ is a surjective open map, it suffices to show that $\Good(M_R)$
is open in $\Spec R$.
So we may assume that $S=X=\Spec R$ is affine and $G$ is split,
and we are to prove that $\Good(N)$ is open for an $R$-finite
$G$-module $N$.

As $N$ is $R$-finite, there exists some finite poset ideal $\pi$ of $X^+$
to which $N$ belongs.
Then $\Ext_G^i(\Delta_G(\lambda),N)=0$ for 
$\lambda\in X^+\setminus \pi$ and $i\geq 0$.
Set $L:=\bigoplus_{\lambda\in\pi}\Delta_G(\lambda)$.
Then $\Good(N)$ is nothing but the complement of the
support of the $R$-module
$\Ext^1_G(L,N)$ by \cite[(III.2.3.8)]{Hashimoto5}.
As $\Ext^1_G(L,N)$ is $R$-finite by \cite[(III.2.3.19)]{Hashimoto5},
the support of $\Ext^1_G(L,N)$ is closed, and we are done.
\end{proof}

\paragraph
Let the notation be as in (\ref{reductive-scheme.par}).
For a quasi-coherent $(G,\O_X)$-module $M$, the good dimension $\GD(M)$ is
defined to be $-\infty$ if $M=0$.
If $M\neq 0$ and there is an exact sequence
\begin{equation}\label{good-sequence.eq}
0\rightarrow M\rightarrow N_0\rightarrow \cdots\rightarrow N_s\rightarrow 0
\end{equation}
such that each $N_i$ is good,
then $\GD(M)$ is defined to be the smallest $s$ such that such an
exact sequence exists.
If there is no such an exact sequence, $\GD(M)$ is defined to be $\infty$.

\paragraph
Assume that $X=\Spec R$ is affine and $G$ is split reductive.
For a $G$-module $M$,
\[
\GD(M)=\sup\{i\mid \bigoplus_{\lambda\in X^+}\Ext^i_G(\Delta_G(\lambda),M)
\neq 0\}.
\]

Note that $M$ is good if and only if $\GD(M)\leq 0$.
If $r\geq 0$, $s\geq -1$, and
\[
0\rightarrow M\rightarrow M_s\rightarrow\cdots M_0\rightarrow N\rightarrow 0
\]
is an exact sequence of $G$-modules with $\GD(M_i)\leq i+r$, then
$\GD(M)\leq s+r+1$ if and only if $\GD(N)\leq r$.

If $M$ and $N$ are good and $M$ is $R$-finite $R$-projective, 
then $M\otimes N$ is good, see \cite[(III.4.5.10)]{Hashimoto5}.
Moreover, if $M$ is $R$-finite $R$-projective with $\GD(M)\leq s$,
then $M$ has an exact sequence of the form (\ref{good-sequence.eq})
such that each $N_i$ is $R$-finite $R$-projective and good.
Indeed, $M$ belongs to some finite poset ideal $\pi$ of $X^+$,
and when we truncate the cobar resolution of $M$ as a $C_\pi$-comodule,
then we obtain such a sequence.

It follows that for an $R$-finite $R$-projective $G$-module $M$,
$\GD(M)\leq s$ if and only if $\GD(\kappa(\fm)\otimes_R M)\leq s$
for any maximal ideal $\fm$ of $R$ by \cite[(III.4.1.8)]{Hashimoto5}.

It also follows that 
if $\GD(M)\leq s$ and $\GD(N)\leq t$ with $M$ being $R$-finite
$R$-projective, then $\GD(M\otimes N)\leq s+t$.

\begin{lemma}\label{sym-finite.thm}
Let $V$ be an $R$-finite $R$-projective 
$G$-module with $\rank V\leq n<\infty$.
Then the following are equivalent.
\begin{description}
\item[(1)] $\Sym V$ is good.
\item[(2)] $\bigoplus_{i=1}^{n-1}\Sym_iV$ is good.
\item[(3)] For $i=1,\ldots, n-1$, $\GD(\ext^iV)\leq i-1$.
\item[(4)] For $i\geq 1$, $\GD(\ext^iV)\leq i-1$.
\end{description}
\end{lemma}

\begin{proof}
We may assume that $R$ is a field.

{\bf (1)$\Rightarrow$(2)} is trivial.

{\bf (2)$\Rightarrow$(3)} We use the induction on $i$.

By assumption and the
induction assumption, $\GD(\Sym_{i-j}V\otimes \ext^j V)\leq j-1$ 
for $j=1,\ldots,i-1$.
On the other hand, $\Sym_i V$ is good.
So by the exact sequence
\begin{equation}\label{Koszul.eq}
0\rightarrow \ext^iV\rightarrow \Sym_1V\otimes \ext^{i-1}V
\rightarrow\cdots\rightarrow \Sym_{i-1}V\otimes \ext^1V
\rightarrow \Sym_iV\rightarrow 0,
\end{equation}
$\GD(\ext^iV)\leq i-1$.

{\bf(3)$\Rightarrow$(4)} is trivial,
as $\dim \ext^iV\leq 1$ for $i\geq n$.

{\bf (4)$\Rightarrow$(1)}
Note that $\Sym_0V=R$ is good.
Now use induction on $i\geq 1$ to prove that $\Sym_iV$ is
good (use the exact sequence (\ref{Koszul.eq}) again).
\end{proof}

\begin{theorem}\label{good-open.thm}
Let $S$ be a scheme, $G$ a reductive $S$-group acting trivially on
a Noetherian $S$-scheme $X$.
Let $M$ be a locally free coherent $(G,\O_X)$-module.
Then
\begin{multline}\label{good-eq.eq}
\Good(\Sym M)=
\{x\in X\mid \Sym (\kappa(x)\otimes_{\Cal O_{X,x}} M_x) \\
\text{ is a good $(\Spec \kappa(x)\times_S G)$-module}\},
\end{multline}
and $\Good(\Sym M)$ is Zariski open in $X$.
\end{theorem}

\begin{proof}
Take a faithfully flat morphism of finite type
$f:\Spec R\rightarrow X$ such that $\Spec R\times_S G$
is split reductive.
Note that $f$ is a surjective open map, and
$f^{-1}(\Good(\Sym M))=
\Good(\Sym f^*M)$.

First we prove that $\Good(\Sym M)$ is open.
We may assume that $S=X=\Spec R$ is affine, and
$G$ is split reductive.

Then by Lemma~\ref{sym-finite.thm} and Lemma~\ref{finite-open.thm},
$$\Good(\Sym M)=\Good(\bigoplus_{i=1}^n \Sym_i M)$$
is open, where the rank of $M$ is less than or equal to $n$.

Next we prove that the equality (\ref{good-eq.eq}) holds.
Let $P\in\Spec R$, and $x=f(P)$.
Then $\Sym(\kappa(x)\otimes_{\O_{X,x}}M_x)$ is good if and only if
$\Sym(\kappa(P)\otimes_{R_P}\Gamma(\Spec R,f^*M)_P)$ is good.
So we may assume that $S=X=\Spec R$ is affine, and $G$ is split reductive.
Let $N$ be an $R$-finite $R$-projective $G$-module of rank at most $n$.
Then $(\Sym N)_P$ is good if and only if $(\bigoplus_{i=1}^n\Sym_i N)_P$ 
is good by Lemma~\ref{sym-finite.thm}.
By \cite[(III.4.1.8)]{Hashimoto5}, 
$(\bigoplus_{i=1}^n\Sym_i N)_P$ is good if and only if
$\kappa(P)\otimes_{R_P}(\bigoplus_{i=1}^n\Sym_iN)_P$ is good.
By Lemma~\ref{sym-finite.thm} again,
it is good if and only if 
$\kappa(P)\otimes_{R_P}(\Sym N)_P$ is so.
Thus the equality (\ref{good-eq.eq}) was proved.
\end{proof}

\begin{corollary}
Let $R$ be a Noetherian domain of characteristic zero,
and $G$ a reductive group over $R$.
If $M$ is an $R$-finite $R$-projective $G$-module,
then $\{P\in\Spec R\mid \Sym(\kappa(P)\otimes_R M) \text{ is good}\}$ 
is a dense open subset of $\Spec R$.
\end{corollary}

\begin{proof}
By Theorem~\ref{good-open.thm},
it suffices to show that $\Good(\Sym M)$ is non-empty.
But the generic point $\eta$ of $\Spec R$ is in
$\Good(\Sym M)$.
Indeed, $\kappa(\eta)$ is a field of characteristic
zero, and any $\kappa(\eta)\otimes_R G$-module is good.
\end{proof}

\end{document}